\documentclass{amsart}%
\usepackage{amsfonts}
\usepackage{amsmath}
\usepackage{amssymb}
\usepackage{graphicx}%
\providecommand{\U}[1]{\protect\rule{.1in}{.1in}}
\newtheorem{theorem}{Theorem}
\theoremstyle{plain}

\newtheorem{corollary}{Corollary}

\newtheorem{definition}{Definition}
\newtheorem{example}{Example}

\newtheorem{lemma}{Lemma}

\newtheorem{proposition}{Proposition}
\newtheorem{remark}{Remark}

\numberwithin{equation}{section}
\begin{document}
\title[ON WEAKLY $S$-PRIME SUBMODULES]{ON WEAKLY $S$-PRIME SUBMODULES}
\author{Hani A. Khashan}
\address{Department of Mathematics, Faculty of Science, Al al-Bayt University, Al
Mafraq, Jordan.}
\email{hakhashan@aabu.edu.jo.}
\author{Ece Yetkin Celikel}
\address{Department of Basic Sciences, Faculty of Engineering, Hasan Kalyoncu
University, Gaziantep, Turkey.}
\email{ece.celikel@hku.edu.tr, yetkinece@gmail.com.}
\thanks{This paper is in final form and no version of it will be submitted for
publication elsewhere.}
\subjclass[2010]{ Primary 13A15, 16P40, Secondary 16D60.}
\keywords{$S$--prime ideal, weakly $S$-prime ideal, $S$-prime submodule, weakly
$S$-prime submodule, amalgamated algebra}

\begin{abstract}
Let $R$ be a commutative ring with a non-zero identity, $S$ be a
multiplicatively closed subset of $R$ and $M$ be a unital $R$-module. In this
paper, we define a submodule $N$ of $M$ with $(N:_{R}M)\cap S=\phi$ to be
weakly $S$-prime if there exists $s\in S$ such that whenever $a\in R$ and
$m\in M$ with $0\neq am\in N$, then either $sa\in(N:_{R}M)$ or $sm\in N$. Many
properties, examples and characterizations of weakly $S$-prime submodules are
introduced, especially in multiplication modules. Moreover, we investigate the
behavior of this structure under module homomorphisms, localizations, quotient
modules, cartesian product and idealizations. Finally, we define two kinds of
submodules of the amalgamation module along an ideal and investigate
conditions under which they are weakly $S$-prime.

\end{abstract}
\maketitle

\section{Introduction}

Throughout this paper, unless otherwise stated, $R$ denotes a commutative ring
with non-zero identity and $M$ is a unital $R$-module. It is well-known that a
proper submodule $N$ of $M$ is called prime if $rm\in N$ for $r\in R$ and
$m\in M$ implies $r\in(N:_{R}M)$ or $m\in N$ where $(N:_{R}M)=\{r\in R$ $|$
$rM\subseteq N\}$. Since prime ideals and submodules have a vital role in ring
and module theory, several generalizations of these concepts have been studied
extensively by many authors (see, for example, \cite{WS-prime}, \cite{wp},
\cite{S-prime}, \cite{2-abs}, \cite{S-prime subm}, \cite{WS-primary}).

In 2007, Atani and Farzalipour introduced the concept of weakly prime
submodules as a generalization of prime submodules. Following \cite{Wp-subm},
a proper submodule $N$ of $M$ is said to be weakly prime if for $r\in R$ and
$m\in M$, whenever $0\neq rm\in N$, then $r\in(N:_{R}M)$ or $m\in N$ . In 2019
a new kind of generalizations of prime submodules has been introduced and
studied by \c{S}engelen Sevim et. al. \cite{S-prime subm}. For a
multiplicatively closed subset $S$ of $R$, they called a proper submodule $N$
of an $R$-module $M$ with $(N:_{R}M)\cap S=\emptyset$ an $S$-prime if there
exists $s\in S$ such that for $r\in R$ and $m\in M$, whenever $rm\in N$, then
either $sr\in(N:_{R}M)$ or $sm\in N$. In particular, an ideal $I$ of $R$ is
called an $S$-prime ideal if $I$ is an $S$-prime submodule of an $R$-module
$R$, \cite{S-prime}. Recently, Almahdi et. al. generalized $S$-prime ideals by
defining the notion of weakly $S$-prime ideals. A proper ideal $I$ of $R$
disjoint with $S$ is said to be weakly $S$-prime if there exists $s\in S$ such
that for $a,b\in R$ and $0\neq ab\in I$, then either $sa\in I$ or $sb\in I$,
\cite{WS-prime}.

Our objective in this paper is to define and study the concept of weakly
$S$-prime submodules as an extension of the above concepts. Let $S$ be a
multiplicatively closed subset of $R$. We call a submodule $N$ of an
$R$-module $M$ with $(N:_{R}M)\cap S=\emptyset$ a weakly $S$-prime submodule
if there exists $s\in S$ such that for $a\in R$ and $m\in M$, whenever $0\neq
am\in N$ then either $sa\in(N:_{R}M)$ or $sm\in N$. In the second section, we
obtain many equivalent statements to characterize this class of submodules
(see Theorems \ref{char} and \ref{fm}), particularly in multiplication modules
(Theorem \ref{IM}). Moreover, various properties of weakly $S$-prime
submodules are considered and many examples are given for supporting the
results (see for example Theorem \ref{(N:M)}, Propositions \ref{(I:s)},
\ref{IN} and Examples \ref{e1}, \ref{ex11}). We investigate the behavior of
this structure under module homomorphisms, localizations, quotient modules,
cartesian product of modules (see Propositions \ref{loc}, \ref{f}, Theorem
\ref{cart} and Corollary \ref{quot}). Let $S$ be a multiplicatively closed
subset of $R$, $M$ be an $R$-module and consider the idealization ring
$R\ltimes M$. For any submodule $K$ of $M$, the set $S\ltimes K=\{(s,k):$
$s\in S$, $k\in K\}$ is a multiplicatively closed subset of $R\ltimes M$. In
Theorem \ref{Ideal}, we justify the relation among weakly $S$-prime ideals of
$R$, weakly $S$-prime submodules of $M$ and weakly $S\ltimes K$-prime ideals
of the the idealization ring $R\ltimes M$.

Let $f:R_{1}\rightarrow R_{2}$ be a ring homomorphism, $J$ be an ideal of
$R_{2}$, $M_{1}$ be an $R_{1}$-module, $M_{2}$ be an $R_{2}$-module (which is
an $R_{1}$-module induced naturally by $f$) and $\varphi:M_{1}\rightarrow
M_{2}$ be an $R_{1}$-module homomorphism. The subring $R_{1}\Join
^{f}J=\left\{  (r,f(r)+j):r\in R_{1}\text{, }j\in J\right\}  $ of $R_{1}\times
R_{2}$ is called the amalgamation of $R_{1}$ and $R_{2}$ along $J$ with
respect to $f$, \cite{Danna}. The amalgamation of $M_{1}$ and $M_{2}$ along
$J$ with respect to $\varphi$ is defined in \cite{Rachida} as%

\[
M_{1}\Join^{\varphi}JM_{2}=\left\{  (m_{1},\varphi(m_{1})+m_{2}):m_{1}\in
M_{1}\text{ and }m_{2}\in JM_{2}\right\}
\]
which is an $(R_{1}\Join^{f}J)$-module with the scaler product defined as
\[
(r,f(r)+j)(m_{1},\varphi(m_{1})+m_{2})=(rm_{1},\varphi(rm_{1})+f(r)m_{2}%
+j\varphi(m_{1})+jm_{2})
\]
For submodules $N_{1}$ and $N_{2}$ of $M_{1}$ and $M_{2}$, respectively, the
sets
\[
N_{1}\Join^{\varphi}JM_{2}=\left\{  (m_{1},\varphi(m_{1})+m_{2})\in M_{1}%
\Join^{\varphi}JM_{2}:m_{1}\in N_{1}\right\}
\]
and
\[
\overline{N_{2}}^{\varphi}=\left\{  (m_{1},\varphi(m_{1})+m_{2})\in M_{1}%
\Join^{\varphi}JM_{2}:\text{ }\varphi(m_{1})+m_{2}\in N_{2}\right\}
\]
are submodules of $M_{1}\Join^{\varphi}JM_{2}$. Section 3 is devoted for
studying several conditions under which the submodules $N_{1}\Join^{\varphi
}JM_{2}$ and $\overline{N_{2}}^{\varphi}$ of $M_{1}\Join^{\varphi}JM_{2}$ are
(weakly) $S$-prime submodules, (see Theorems \ref{Amalg}, \ref{Amalg3}).
Furthermore, we conclude some particular results for the duplication of a
module along an ideal (see Corollaries \ref{ca1}-\ref{ca3}, \ref{Dup}%
-\ref{Dup2} and Theorem \ref{Amalg2}).

For the sake of completeness, we start with some definitions and notations
which will be used in the sequel. A non-empty subset $S$ of a ring $R$ is said
to be a multiplicatively closed set if $S$ is a subsemigroup of $R$ under
multiplication. An $R$-module $M$ is called multiplication provided for each
submodule $N$ of $M,$ there exists an ideal $I$ of $R$ such that $N=IM$. In
this case, $I$ is said to be a presentation ideal of $N$. In particular, for
every submodule $N$ of a multiplication module $M$, $ann(M/N)=(N:_{R}M)$ is a
presentation for $N$. The product of two submodules $N$ and $K$ of a
multiplication module $M$ is defined as $NK=(N:_{R}M)(K:_{R}M)M$. For
$m_{1},m_{2}\in M$, by $m_{1}m_{2},$ we mean the product of $Rm_{1}$ and
$Rm_{2}$ which is equal to $IJM$ for presentation ideals $I$ and $J$ of
$m_{1}$ and $m_{2}$, respectively \cite{Ameri}. Let $N$ be a proper submodule
of an $R$-module $M$. The radical of $N$ (denoted by $M$-$rad(N)$) is defined
in \cite{El-bast} to be the intersection of all prime submodules of $M$
containing $N$. If $M$ is multiplication, then $M$-$rad(N)=\{m\in M$ $|$
$m^{k}\subseteq N$ for some $k\geq0\}.$ As usual, $%
\mathbb{Z}
$, $%
\mathbb{Z}
_{n}$ and $%
\mathbb{Q}
$ denotes the ring of integers, the ring of integers modulo $n$ and the field
of rational numbers, respectively. For more details and terminology, one may
refer to \cite{Ali}, \cite{Majed}, \cite{At}, \cite{Gilmer}, \cite{Larsen}.

\section{Characterizations of Weakly $S$-prime Submodules}

We begin with the definitions and relationships of the main concepts of the paper.

\begin{definition}
Let $S$ be a multiplicatively closed subset of a ring $R$ and $N$ be a
submodule of an $R$-module $M$ with $(N:_{R}M)\cap S=\emptyset$. We call $N$ a
weakly $S$-prime submodule if there exists (a fixed) $s\in S$ such that for
$a\in R$ and $m\in M$, whenever $0\neq am\in N$ then either $sa\in(N:_{R}M)$
or $sm\in N$. The fixed element $s\in S$ is said to be a weakly $S$-element of
$N.$
\end{definition}

It is clear that every $S$-prime submodule is a weakly $S$-prime submodule.
Since the zero submodule is (by definition) a weakly $S$-prime submodule of
any $R$-module, then the converse is not true in general. For a less trivial
example, let $M$ be a non-zero local multiplication $R$-module with the unique
maximal submodule $K$ such that $(K:_{R}M)K=0$. If we consider $S=\left\{
1_{R}\right\}  $, then every proper submodule of $M$ is weakly $S$-prime,
\cite{Ali}. Hence, there is a weakly $S$-prime submodule in $M$ that is not
$S$-prime.

Also, every weakly prime submodule $N$ of an $R$-module $M$ satisfying
$(N:_{R}M)\cap S=\emptyset$ is a weakly $S$-prime submodule of $M$ and the two
concepts coincide if $S\subseteq U(R)$ where $U(R)$ denotes the set of units
in $R$. The following example shows that the converse need not be true.

\begin{example}
\label{e1}Consider the $%
\mathbb{Z}
$-module $M=%
\mathbb{Z}
\times%
\mathbb{Z}
_{6}$ and let $N=2%
\mathbb{Z}
\times\left\langle \bar{3}\right\rangle $. Then $N$ is a (weakly) $S$-prime
submodule of $M$ where $S=\left\{  2^{n}:n\in%
\mathbb{N}
\cup\left\{  0\right\}  \right\}  $. Indeed, let $(0,\bar{0})\neq
r.(r^{\prime},m)\in N$ for $r,r^{\prime}\in%
\mathbb{Z}
$ and $m\in%
\mathbb{Z}
_{6}$ such that $2r\notin(N:M)=6%
\mathbb{Z}
$. Then $r.m\in\left\langle \bar{3}\right\rangle $ with $r\notin3%
\mathbb{Z}
$ and so $m\in\left\langle \bar{3}\right\rangle $. Thus, $2.(r^{\prime},m)\in
N$ as needed. On the other hand, $N$ is not a weakly prime submodule since
$(0,\bar{0})\neq2.(1,\bar{0})\in N$ but $2\notin(N:M)$ and $(1,\bar{0})\notin
N$.
\end{example}

Let $N$ be a submodule of an $R$-module $M$ and $I$ be an ideal of $R$. The
residual of $N$ by $I$ is the set $(N:_{M}I)=\{m\in M:Im\subseteq N\}$. It is
clear that $(N:_{M}I)$ is a submodule of $M$ containing $N$. More generally,
for any subset $A\subseteq R$, $(N:_{M}A)$ is a submodule of $M$ containing
$N$.

\begin{theorem}
\label{char}Let $S$ be a multiplicatively closed subset of a ring $R$. For a
submodule $N$ of an $R$-module $M$ with $(N:_{R}M)\cap S=\emptyset$, the
following conditions are equivalent.
\end{theorem}

\begin{enumerate}
\item $N$ is a weakly $S$-prime submodule of $M$.

\item There exists $s\in S$ such that $(N:_{M}a)=(0:_{M}a)$ or $(N:_{M}%
a)\subseteq(N:_{M}s)$ for each $a\notin(N:_{R}sM)$.

\item There exists $s\in S$ such that for any $a\in R$ and for any submodule
$K$ of $M$, if $0\neq aK\subseteq N$, then $sa\subseteq(N:_{R}M)$ or
$sK\subseteq N.$

\item There exists $s\in S$ such that for any ideal $I$ of $R$ and a submodule
$K$ of $M$, if $0\neq IK\subseteq N$, then $sI\subseteq(N:_{R}M)$ or
$sK\subseteq N.$
\end{enumerate}

\begin{proof}
(1)$\Rightarrow$(2). Let $s\in S$ be a weakly $S$-element of $N$ and
$a\notin(N:_{M}sM).$ Let $m\in(N:_{M}a)$. If $am=0$, then clearly $m\in
(0:_{M}a)$. If $0\neq am\in N$, then, we conclude $sm\in N$ as $sa\notin
(N:_{R}M)$ and $N$ is weakly $S$-prime in $M$. Thus, $m\in(N:_{M}s)$ and so
$(N:_{M}a)\subseteq(0:_{M}a)\cup(N:_{M}s)$. Therefore, $(N:_{M}a)\subseteq
(0:_{M}a)$ (which implies $(N:_{M}a)=(0:_{M}a)$) or $(N:_{M}a)\subseteq
(N:_{M}s)$.

(2)$\Rightarrow$(3). Choose $s\in S$ as in (2) and suppose $0\neq aK\subseteq
N$ and $sa\notin(N:_{R}M)$ for some $a\in R$ and a submodule $K$ of $M$. Then
$K\subseteq(N:_{M}a)\backslash(0:_{M}a)$ and by (2) we get $K\subseteq
(N:_{M}a)\subseteq(N:_{M}s)$. Thus, $sK\subseteq N$ as required.

(3)$\Rightarrow$(4). Choose $s\in S$ as in (3) and suppose $0\neq IK\subseteq
N$ and $sI\nsubseteq(N:_{R}M)$ for some ideal $I$ of $R$ and a submodule $K$
of $M.$ Then there exists $a\in I$ with $sa\notin(N:_{R}M)$. If $aK\neq0$,
then by (3), we have $sK\subseteq N$ as needed. Assume that $aK=0$. Since
$IK\neq0$, there is some $b\in I$ with $bK\neq0$. If $sb\notin(N:_{R}M),$ then
from (3), we have $sK\subseteq N$. Now, assume that $sb\in(N:_{R}M).$ Since
$sa\notin(N:_{R}M),$ we have $s(a+b)\notin(N:_{R}M)$. Hence, $0\neq
(a+b)K\subseteq N$ implies $sK\subseteq I$ again by (3) and we are done.

(4)$\Rightarrow$(1). Let $a\in R,$ $m\in M$ with $0\neq am\in N$. The result
follows directly by taking $I=aR$ and $K=<m>$ in (4).
\end{proof}

\begin{theorem}
\label{fm}Let $M$ be a faithful multiplication $R$-module and $S$ be a
multiplicatively closed subset of $R$. The following are equivalent.
\end{theorem}

\begin{enumerate}
\item $N$ is a weakly $S$-prime submodule of $M$.

\item $N\cap SM=\phi$ and there exists $s\in S$ such that whenever $K,L$ are
submodules of $M$ and $0\neq KL\subseteq N$, then $sK\subseteq N$ or
$sL\subseteq N$.
\end{enumerate}

\begin{proof}
Clearly, we have $N\cap SM=\phi$ if and only if $(N:_{R}M)\cap S=\emptyset$.

(1)$\Rightarrow$(2). Let $I$ be a presentation ideal of $K$ and $s$ be a
weakly $S$-element of $N$. Then $0\neq IL\subseteq N$ yields that either
$sI\subseteq(N:_{R}M)$ or $sL\subseteq N$ by Theorem \ref{char}. Hence,
$sK=sIM\subseteq N$ or $sL\subseteq N$, as needed.

(2)$\Rightarrow$(1). Let $s\in S$ be as in (2) and suppose $0\neq IL\subseteq
N$ for some ideal $I$ of $R$ and submodule $L$ of $M$. Put $K=IM$ and assume
that $sL\nsubseteq N$. Then $0\neq KL\subseteq N$ which implies $sK\subseteq
N$. Therefore, $sI\subseteq(N:_{R}M)$ and the result follows by Theorem
\ref{char}.
\end{proof}

\begin{theorem}
\label{(N:M)}Let $N$ be a submodule of an $R$-module $M$ and $S$ be a
multiplicatively closed subset of $R$. The following statements hold.
\end{theorem}

\begin{enumerate}
\item If $N$ is a weakly $S$-prime submodule of $M$, then for every submodule
$K$ with $(N:_{R}K)\cap S=\emptyset$ and $Ann(K)=0,$ $(N:_{R}K)$ is a weakly
$S$-prime ideal of $R$. In particular, if $M$ is faithful, then $(N:_{R}M)$ is
a weakly $S$-prime ideal of $R.$

\item If $M$ is multiplication and $(N:_{R}M)$ is a weakly $S$-prime ideal of
$R$, then $N$ is a weakly $S$-prime submodule of $M$.

\item If $M~$is faithful multiplication and $I$ is an ideal of $R$, then $I$
is weakly $S$-prime in $R$ if and only if $IM$ is a weakly $S$-prime submodule
of $M$.

\item If $N$ is a weakly $S$-prime submodule of $M$ and $A$ is a subset of $R$
such that $(0:_{M}A)=0$ and $Z_{(N:_{R}M)}(R)\cap A=\phi$, then $(N:_{M}A)$ is
a weakly $S$-prime submodule of $M$.
\end{enumerate}

\begin{proof}
(1) Suppose $s\in S$ is a weakly $S$-element of $N$ and let $a,b\in R$ with
$0\neq ab\in(N:_{R}K)$. Since $Ann(K)=0$, we have $0\neq abK\subseteq N$ which
implies $sa\in(N:_{R}M)$ or $sbK\subseteq N$ by Theorem \ref{char}. Hence,
$sa\in(N:_{R}K)$ or $sb\in(N:_{R}K)$. Thus, $(N:_{R}K)$ is a weakly $S$-prime
ideal associated with the same $s\in S$. The "in particular" part is clear.

(2) Suppose $M$ is multiplication and $(N:_{R}M)$ is a weakly $S$-prime ideal
of $R$. Let $I$ be an ideal of $R$ and $K$ be a submodule of $M$ with $0\neq
IK\subseteq N$. Since $M$ is multiplication, we may write $K=JM$ for some
ideal $J$ of $R$. Thus, $0\neq IJ\subseteq(N:_{R}M)$, and by \cite[Theorem
1]{S-prime}, there exists an $s\in S$ such that $sI\subseteq(N:_{R}M)$ or
$sJ\subseteq(N:_{R}M)$. Thus, $sI\subseteq(N:_{R}M)$ or $sK=sJM\subseteq
(N:_{R}M)M=N$. Therefore, $N$ is a weakly $S$-prime submodule of $M$ by (4) of
Theorem \ref{char}.

(3) Suppose $M~$is faithful multiplication and $I$ is an ideal of $R$. Since
$(IM:_{R}M)=I,$ the result follows from (1) and (2).

(4) Let $s\in S$ be a weakly $S$-element of $N$. We firstly note that
$((N:_{M}A):_{R}M)\cap S=\phi$. Indeed, if $t\in((N:_{M}A):_{R}M)\cap S$, then
$tA\subseteq(N:_{R}M)$ and so $t\in(N:_{R}M)$ as $Z_{(N:_{R}M)}(R)\cap A=\phi
$, a contradiction. Let $r\in R$ and $m\in M$ such that $0\neq rm\in(N:_{M}%
A)$. Then $0\neq Arm\subseteq N$ since $(0:_{M}A)=0$. By assumption, either
$sr\in(N:_{R}M)$ or $sAm\subseteq N$. Thus, $sr\in((N:_{M}A):_{R}M)$ or
$sm\in(N:_{M}A)$ as needed.
\end{proof}

We show by the following example that the condition "faithful module" in
Theorem \ref{(N:M)} (1) is crucial.

\begin{example}
Let $p_{1},p_{2}$ and $p_{3}$ be distinct prime integers. Consider the
non-faithful $%
\mathbb{Z}
$-module $M=%
\mathbb{Z}
_{p_{1}p_{2}}\times%
\mathbb{Z}
_{p_{1}p_{2}}$ and the multiplicatively closed subset $S=\left\{  p_{3}%
^{n}:n\in%
\mathbb{N}
\cup\left\{  0\right\}  \right\}  $ of $%
\mathbb{Z}
$. While $N=\overline{0}\times\overline{0}$ is a weakly $S$-prime submodule of
$M$, we have clearly $(N:_{%
\mathbb{Z}
}M)=\left\langle p_{1}p_{2}\right\rangle $ is not a weakly $S$-prime ideal of
$%
\mathbb{Z}
$.
\end{example}

Let $N$ be a proper submodule of an $R$-module $M$. Then $N$ is said to be a
maximal weakly $S$-prime submodule if there is no weakly $S$-prime submodule
which contains $N$ properly. In the following corollary, by $Z(M)$, we denote
the set $\{r\in R:rm=0$ for some $m\in M\backslash\{0_{M}\}\}$.

\begin{corollary}
Let $N$ be a submodule of $M$ such that $Z_{(N:_{R}M)}(R)\cup Z(M)\subseteq
(N:_{R}M)$.\ If $N$ is a maximal weakly $S$-prime submodule of $M$, then $N$
is an $S$-prime submodule of $M$.
\end{corollary}

\begin{proof}
Let $s\in S$ be a weakly $S$-element of $N$. Suppose that $am\in N$ and
$sa\notin(N:_{R}M)$ for some $a\in R$ and $m\in M$. Since $a\notin(N:_{R}M)$,
then by assumption, $a\notin Z_{(N:_{R}M)}(R)$ and $(0:_{M}a)=0$. It follows
by Theorem \ref{(N:M)} (4) that $(N:_{M}a)$ is a weakly $S$-prime submodule of
$M$. Therefore, $sm\in(N:_{M}a)=N$ by the maximality of $N$ and so $N$ is an
$S$-prime submodule of $M$.
\end{proof}

As $N=(N:M)M$ for any submodule $N$ of a multiplication $R$-module $M,$ we
have the following consequence of Theorem \ref{(N:M)}.

\begin{theorem}
\label{IM}Let $M$ be a faithful multiplication $R$-module and $N$ be a
submodule of $M$. The following are equivalent.
\end{theorem}

\begin{enumerate}
\item $N$ is a weakly $S$-prime submodule of $M$.

\item $(N:_{R}M)$ is a weakly $S$-prime ideal of $R.$

\item $N=IM$ for some weakly $S$-prime ideal $I$ of $R.$
\end{enumerate}

For a next result, we need to recall the following lemma.

\begin{lemma}
\cite{Majed}\label{Majed} For an ideal $I$ of a ring $R$ and a submodule $N$
of a finitely generated faithful multiplication $R$-module $M$, the following hold.

\begin{enumerate}
\item $(IN:_{R}M)=I(N:_{R}M)$.

\item If $I$ is finitely generated faithful multiplication, then

\begin{enumerate}
\item $(IN:_{M}I)=N$.

\item Whenever $N\subseteq IM$, then $(JN:_{M}I)=J(N:_{M}I)$ for any ideal $J$
of $R$.
\end{enumerate}
\end{enumerate}
\end{lemma}

\begin{proposition}
\label{IN}Let $I$ be a finitely generated faithful multiplication ideal of a
ring $R,$ $S$ a multiplicatively closed subset of $R$ and $N$ a submodule of a
finitely generated faithful multiplication $R$-module $M$.
\end{proposition}

\begin{enumerate}
\item If $IN$ is a weakly $S$-prime submodule of $M$ and $(N:_{R}M)\cap
S=\phi$, then either $I$ is a weakly $S$-prime ideal of $R$ or $N$ is a weakly
$S$-prime submodule of $M$.

\item $N$ is a weakly $S$-prime submodule of $IM$ if and only if $(N:_{M}I)$
is a weakly $S$-prime submodule of $M$.
\end{enumerate}

\begin{proof}
(1) Let $s\in S$ be weakly $S$-element of $IN$. Suppose $N=M$. In this case,
$I=I(N:_{R}M)=(IN:_{R}M)$ is a weakly $S$-prime ideal of $R$ by Theorem
\ref{IM}. Now, suppose that $N$ is proper. Hence, Lemma \ref{Majed} implies
$N=(IN:_{M}I)$ and so we conclude that $(N:_{R}M)=((IN:_{M}I):_{R}%
M)=(I(N:_{R}M):_{M}I)$. Suppose $a\in R,$ $m\in M$ such that $0\neq am\in N$
and $sa\notin(N:_{R}M).$ Since $I$ is faithful, then $(0:_{M}I)=Ann_{R}%
(I)M=0$, \cite{Majed} and so $0\neq Iam\subseteq IN.$ Since clearly
$sa\notin(IN:_{R}M)$ and $IN$ is a weakly $S$-prime submodule, $sIm\subseteq
IN$ by Theorem \ref{char}. By Lemma \ref{Majed} (2), we have $sm\in
(IN:_{M}I)=N$, and thus $N$ is a weakly $S$-prime submodule of $M$.

(2) Suppose $N$ is a weakly $S$-prime submodule of $IM$ with a weakly
$S$-element $s^{\prime}\in S$. Then $((N:_{M}I):_{R}M)\cap S=(N:_{R}IM)\cap
S=\phi$. Let $a\in R$ and $m\in M$ with $0\neq am\in$ $(N:_{M}I)$ and
$s^{\prime}a\notin((N:_{M}I):_{R}M)=(N:_{R}IM)$. If $amI=0$, then $am\in
(0_{M}:I)=Ann_{R}(I)M=0$, a contradiction. Thus, $0\neq amI\subseteq N$. Since
$N$ is a weakly $S$-prime submodule of $IM,$ Theorem \ref{char} yields that
$s^{\prime}mI\subseteq N$, and so $s^{\prime}m\in(N:_{M}I)$, as required.

Conversely, suppose $(N:_{M}I)$ is a weakly $S$-prime submodule of $M$ with a
weakly $S$-element $s^{\prime}\in S$. Then $(N:_{R}IM)\cap S=((N:_{M}%
I):_{R}M)\cap S=\phi$. Now, let $a\in R$ and $m^{\prime}\in IM$ such that
$0\neq am^{\prime}\in N$ and $s^{\prime}a\notin(N:_{R}IM)=((N:_{M}I):_{R}M)$.
Then $a(\left\langle m^{\prime}\right\rangle :_{M}I)=(\left\langle am^{\prime
}\right\rangle :_{M}I)\subseteq(N:_{M}I)$. If $a(\left\langle m^{\prime
}\right\rangle :_{M}I)=0$, then by (2) of Lemma \ref{Majed}, we have
$am^{\prime}\in a(Im^{\prime}:_{M}I)\subseteq a(\left\langle m^{\prime
}\right\rangle :_{M}I)=0$, a contradiction. Thus, $0\neq a(\left\langle
m^{\prime}\right\rangle :_{M}I)\subseteq$ $(N:_{M}I)$ and so $s^{\prime
}(\left\langle m^{\prime}\right\rangle :_{M}I)\subseteq(N:_{M}I)$ as
$s^{\prime}a\notin((N:_{M}I):_{R}M)$. Again, by Lemma \ref{Majed}, we conclude
that $s^{\prime}m^{\prime}\in(I\left\langle s^{\prime}m^{\prime}\right\rangle
:_{M}I)=Is^{\prime}(\left\langle m^{\prime}\right\rangle :_{M}I)\subseteq
I(N:_{M}I)=(IN:_{M}I)=N$. Therefore, $N$ is a weakly $S$-prime submodule of
$IM.$
\end{proof}

\begin{proposition}
\label{(I:s)}Let $S$ be a multiplicatively closed subset of a ring $R$ and $N$
be a submodule of an $R$-module $M$ such that $(N:_{R}M)\cap S=\phi$. If
$(N:_{M}s)$ is a weakly prime submodule of $M$ for some $s\in S$, then $N$ is
a weakly $S$-prime submodule of $M$. The converse holds for non-zero
submodules $N$ if $S\cap Z(M)=\emptyset.$
\end{proposition}

\begin{proof}
Suppose $(N:_{M}s)$ is a weakly prime submodule of $M$ for some $s\in S$ and
let $a\in R,$ $m\in M$ such that $0\neq am\in N\subseteq(N:_{M}s)$. Then
either $a\in((N:_{M}s):_{R}M)=((N:_{R}M):_{R}s)$ or $m\in(N:_{M}s)$ and so
either $sa\in(N:_{R}M)$ or $sm\in N$ as required. Conversely, suppose
$N\neq0_{M}$ is weakly $S$-prime submodule of $M$ with weakly $S$-element
$s\in S$. Let $a\in R$ and $m\in M$ such that $0\neq am\in(N:_{M}s)$. Since
$S\cap Z(M)=\emptyset$, we have $0\neq sam\in N$ which implies either
$s^{2}a\in(N:_{R}M)$ or $sm\in N$. If $sm\in N$, then $m\in(N:_{M}s)$ and we
are done. Suppose $s^{2}a\in(N:_{R}M).$ If $s^{2}aM=0$, then $s^{2}\in S\cap
Z(M)$, a contradiction. Hence, $0\neq s^{2}aM\subseteq N$ implies either
$s^{3}\in(N:_{R}M)$ or $saM\subseteq N$. But $(N:_{R}M)\cap S=\phi$ implies
$saM\subseteq N$ and so $a\in(N:_{R}sM)=((N:_{M}s):_{R}M)$. Therefore,
$(N:_{M}s)$ is a weakly prime submodule of $M$.
\end{proof}

If $S\cap Z(M)\neq\emptyset$, then the converse of Proposition \ref{(I:s)}
need not be true as we can see in the following example.

\begin{example}
\label{ex11}Consider the $%
\mathbb{Z}
$-module $M=%
\mathbb{Z}
\times%
\mathbb{Z}
_{6}$ and let $N=\left\langle 0\right\rangle \times\left\langle \bar
{0}\right\rangle $. Then $N$ is a weakly $S$-prime submodule of $M$ for
$S=\left\{  3^{n}:n\in%
\mathbb{N}
\right\}  $. Now, for each $n\in%
\mathbb{N}
$, we have clearly $(N:_{M}3^{n})\mathbf{=}\left\langle 0\right\rangle
\mathbf{\times}\left\langle \bar{2}\right\rangle $ which is not a weakly prime
submodule of $M$. Indeed, $2.(0,\bar{1})\in(N:_{M}3^{n})$ but $2\notin
((N:_{M}3^{n}):_{R}M)=\left\langle 0\right\rangle $ and $(0,\bar{1}%
)\notin(N:_{M}3^{n})$. We note that $S\cap Z(M)=S\neq\emptyset$.
\end{example}

\begin{proposition}
\label{p1}Let $M$ be a\ faithful multiplication $R$-module and $S$ be a
multiplicatively closed subset of $R$.
\end{proposition}

\begin{enumerate}
\item If $N$ is a weakly $S$-prime submodule of $M$ that is not $S$-prime,
then $s\sqrt{0_{R}}N=0_{M}$ for some $s\in S$.

\item If $N$ and $K$ are two weakly $S$-prime submodules of $M$ that are not
$S$-prime, then $sNK=0_{M}$ for some $s\in S$.
\end{enumerate}

\begin{proof}
(1) Let $N$ be a weakly $S$-prime submodule of $M$ which is not $S$-prime.
Then by (1) of Theorem \ref{(N:M)} and \cite[Proposition 2.9 (ii)]{S-prime
subm}, $(N:_{R}M)$ is a weakly $S$-prime ideal of $R$ that is not $S$-prime.
Hence, we get $s(N:_{R}M)\sqrt{0_{R}}=0_{R}$ by \cite[Proposition 9]{WS-prime}
and thus, $sN\sqrt{0_{R}}=s(N:_{R}M)M\sqrt{0_{R}}=0_{R}M=0_{M}.$

(2) Since $N$ and $K$ are two weakly $S$-prime submodules that are not
$S$-prime, $(N:_{R}M)$ and $(K:_{R}M)$ are weakly $S$-prime ideals of $R$ that
are not $S$-prime by Theorem \ref{(N:M)} and \cite[Proposition 2.9
(ii)]{S-prime subm}. Hence, there exists some $s\in S$ such that
$s(N:_{R}M)(K:_{R}M)=0_{R}$ by \cite[Corollary 11]{WS-prime} and so $sNK=0.$
\end{proof}

\begin{corollary}
Let $M$ be a\ faithful multiplication $R$-module, $S$ be a multiplicatively
closed subset of a ring $R.$ If $N$ is a weakly $S$-prime submodule of $M$,
then either $N\subseteq\sqrt{0_{R}}M$ or $s\sqrt{0_{R}}M\subseteq N$ for some
$s\in S.$ Additionally, if $R$ is a reduced ring, then $N=0_{M}$ or $N$ is
$S$-prime$.$
\end{corollary}

\begin{proof}
Suppose that $N$ is a weakly $S$-prime submodule of $M$. Then from Theorem
\ref{(N:M)} (1), $(N:_{R}M)$ is a weakly $S$-prime ideal of $R$ and by
\cite[Corollary 6]{WS-prime}, we conclude either $(N:_{R}M)\subseteq
\sqrt{0_{R}}$ or $s\sqrt{0_{R}}\subseteq(N:_{R}M).$ Since $N=(N:_{R}M)M$, we
are done.
\end{proof}

\begin{proposition}
\label{loc}Let $N$ be a submodule of an $R$-module $M$ and $S$ be a
multiplicatively closed subset of $R$ with $Z(M)\cap S=\emptyset$.
\end{proposition}

\begin{enumerate}
\item If $N$ is a weakly $S$-prime submodule of $M$, then $S^{-1}N$ is a
weakly prime submodule of $S^{-1}M$ and there exists an $s\in S$ such that
$(N:_{M}t)\subseteq(N:_{M}s)$ for all $t\in S$.

\item If $M$ is finitely generated, then the converse of (1) holds.
\end{enumerate}

\begin{proof}
(1) Suppose $s\in S$ is a weakly $S$-element of $N$. In proving that $S^{-1}N$
is a weakly prime submodule of $S^{-1}M$ we do not need the assumption
$Z(M)\cap S=\emptyset$. Let $0_{S^{-1}M}\neq\frac{r}{s_{1}}\frac{m}{s_{2}}\in
S^{-1}N$ for some $\frac{r}{s_{1}}\in S^{-1}R$ and $\frac{m}{s_{2}}\in
S^{-1}M.$ Then $urm\in N$ for some $u\in S$. If $urm=0$, then $\frac{rm}%
{s_{1}s_{2}}=\frac{urm}{us_{1}s_{2}}=0_{S^{-1}M}$, a contradiction. Hence,
$0\neq urm\in N$ yields either $sur\in(N:_{R}M)$ or $sm\in N.$ Thus, $\frac
{r}{s_{1}}=\frac{sur}{sus_{1}}\in S^{-1}(N:_{R}M)\subseteq(S^{-1}N:_{S^{-1}%
R}S^{-1}M)$ or $\frac{m}{s_{2}}=\frac{sm}{ss_{2}}\in S^{-1}N$ and so $S^{-1}N$
is a weakly prime submodule of $S^{-1}M.$ Now, let $t\in S$ and $m\in
(N:_{M}t).$ Then $0\neq tm\in N$ as $Z(M)\cap S=\emptyset$ and so
$st\in(N:_{M}M)\cap S$ or $sm\in N.$ Since the first one gives a
contradiction, we have $m\in(N:_{M}s).$ Thus, $(N:_{M}t)\subseteq(N:_{M}s)$
for all $t\in S$.

(2) Suppose $M$ is finitely generated choose $s\in S$ as in (1). If
$(N:_{R}M)\cap S\neq\emptyset$, then clearly $S^{-1}N=S^{-1}M$, a
contradiction. Let $0\neq am\in N$ for some $a\in R$ and $m\in M$. Since
$Z(M)\cap S=\emptyset,$ we have $0\neq\frac{a}{1}\frac{m}{1}\in S^{-1}N$. By
assumption, either $\frac{a}{1}\in(S^{-1}N:_{S^{-1}R}S^{-1}M)=S^{-1}(N:_{R}M)$
as $M$ is finitely generated or $\frac{m}{1}\in S^{-1}N.$ Hence, $va\in
(N:_{R}M)$ for some $v\in S$ or $wm\in N$ for some $w\in S$. If $va\in
(N:_{R}M)$, then our hypothesis implies $aM\subseteq(N:_{M}v)\subseteq
(N:_{M}s)$ and so $sa\in(N:_{R}M).$ If $wm\in N$, then again $m\in
(N:_{M}w)\subseteq(N:_{M}s)$, and so $sm\in N$. Therefore, $N$ is a weakly
$S$-prime submodule of $M.$
\end{proof}

However, $S^{-1}N$ is a weakly prime submodule of $S^{-1}M$ does not imply
that $N$ is a weakly prime submodule of $M$. For example, it was shown in
\cite[Example 2.4]{S-prime subm} that $N=%
\mathbb{Z}
\times\left\{  0\right\}  $ is not a (weakly) $S$-prime submodule of the $%
\mathbb{Z}
$-module $%
\mathbb{Q}
\times%
\mathbb{Q}
$ where $S=%
\mathbb{Z}
\setminus\left\{  0\right\}  $. But $S^{-1}N$ is a weakly prime submodule of
the vector space (over $S^{-1}%
\mathbb{Z}
=%
\mathbb{Q}
)$ $S^{-1}(%
\mathbb{Q}
\times%
\mathbb{Q}
)$.

\begin{remark}
\label{r}Let $M$ be an $R$-module and $S$, $T$ be two multiplicatively closed
subsets of $R$ with $S\subseteq T$. If $N$ is a weakly $S$-prime submodule of
$M$ and $(N:_{R}M)\cap T=\emptyset$, then $N$ is a weakly $T$-prime submodule
of $M.$
\end{remark}

Let $S$ be a multiplicatively closed subset of a ring $R$. The saturation of
$S$ is the set $S^{\ast}=\{x\in R:xy\in S$ for some $y\in R$\}, see
\cite{Gilmer}. It is clear that $S^{\ast}$ is a multiplicatively closed subset
of $R$ and that $S\subseteq S^{\ast}$.

\begin{proposition}
Let $S$ be a multiplicatively closed subset of a ring $R$ and $N$ be a
submodule of an $R$-module $M$ such that $(N:_{R}M)\cap S=\emptyset$. Then $N$
is a weakly $S$-prime submodule of $M$ if and only if $N$ is a weakly
$S^{\ast}$-prime submodule of $M.$
\end{proposition}

\begin{proof}
Let $N$ be a weakly $S^{\ast}$-prime submodule of $M$ with a weakly
$S$-element $s^{\ast}\in S^{\ast}$. Choose $r\in R$ such that $s=s^{\ast}r\in
S$. Suppose $0\neq am\in N$ for some $a\in R$ and $m\in M$. Then either
$s^{\ast}a\in(N:_{R}M)$ or $s^{\ast}m\in N$. Thus, $sa\in(N:_{R}M)$ or $sm\in
N$ and we are done. Conversely, suppose $N$ is weakly $S^{\ast}$-prime. By
using Remark \ref{r}, it is enough to prove that $(N:_{R}M)\cap S^{\ast
}=\emptyset$. Suppose there exists $s^{\ast}\in$ $(N:_{R}M)\cap S^{\ast}$.
Then there is $r\in R$ such that $s=s^{\ast}r\in(N:_{R}M)\cap S$, a contradiction.
\end{proof}

\begin{lemma}
\label{lem}Let $S$ be a multiplicatively closed subset of a ring $R.$ If $I$
is a weakly $S$-prime ideal of $R$ and $\{0_{R}\}$ is an $S$-prime ideal of
$R$, then $\sqrt{I}$ is an $S$-prime ideal of $R.$
\end{lemma}

\begin{proof}
Suppose $I$ is weakly $S$-prime associated to $s_{1}$ and $\{0_{R}\}$ is
$S$-prime associated with $s_{2}$. Since $I\cap S=\emptyset$, we have
$\sqrt{I}\cap S=\emptyset$. Let $a,b\in R$ with $ab\in\sqrt{I}.$ Then
$a^{n}b^{n}\in I$ for some positive integer $n$. If $a^{n}b^{n}\neq0$, then we
have $s_{1}a^{n}\in I$ or $s_{1}b^{n}\in I$ that is $s_{1}a\in\sqrt{I}$ or
$s_{1}b\in\sqrt{I}$. If $a^{n}b^{n}=0$, then by assumption, either $s_{2}%
a^{n}=0$ or $s_{2}b^{n}=0$ and so $s_{2}a\in\sqrt{I}$ or $s_{2}b\in\sqrt{I}.$
Thus, $\sqrt{I}$ is an $S$-prime ideal of $R$ associated with $s=s_{1}s_{2}.$
\end{proof}

\begin{proposition}
\label{S-prime}Let $M$ be a finitely generated faithful multiplication
$R$-module and $S$ be a multiplicatively closed subset of $R$. If $N$ is a
weakly $S$-prime submodule of $M$ and $\{0_{R}\}$ is an $S$-prime ideal of
$R$, then $M$-$rad(N)$ is an $S$-prime submodule of $R$.
\end{proposition}

\begin{proof}
By \cite[Lemma 2.4]{2-abs}, we have $(M$-$rad(N):M)=\sqrt{(N:_{R}M)}$. Since
$N$ is a weakly $S$-prime submodule of $M,$ $(N:_{R}M)$ is so by Theorem
\ref{(N:M)}. By Lemma \ref{lem}, $\sqrt{(N:_{R}M)}$ is an $S$-prime ideal of
$R.\ $Thus, the claim follows from \cite[Proposition 2.9 (ii)]{S-prime
subm}$.$
\end{proof}

\begin{proposition}
\label{int}Let $S$ be a multiplicatively closed subset of a ring $R$. If $N$
is a weakly $S$-prime submodule of an $R$-module $M$, then for any submodule
$K$ of $M$ with $(K:_{R}M)\cap S\neq\emptyset$, $N\cap K$ is a weakly
$S$-prime submodule of $M$. Additionally, if $M$ is multiplication, then $NK$
is a weakly $S$-prime submodule of $M.$
\end{proposition}

\begin{proof}
Note that $(N\cap K:_{R}M)\cap S=\emptyset$ as $(N:_{R}M)\cap S=\emptyset$.
Let $s\in S$ be a weakly $S$-element of $N$ and let $0\neq am\in N\cap
K\subseteq N$. Then $sa\in(N:_{R}M)$ or $sm\in N$. Choose $s^{\prime}%
\in(K:_{R}M)\cap S$. Then $ss^{\prime}a\in(N:_{R}M)\cap(K:_{R}M)=(N\cap
K:_{R}M)$ or $ss^{\prime}m\in N\cap(K:_{R}M)M=N\cap K$. Thus, $N\cap K$ is a
weakly $S$-prime submodule of $M$ with a weakly $S$-element $t=ss^{\prime}.$
Putting in mind that $NK=(N:_{R}M)(K:_{R}M)M,$ the rest of the proof is very similar.
\end{proof}

Notice that if $N$ is weakly prime and $K$ is as above, then $N\cap K$ need
not be weakly prime. For instance, consider the $%
\mathbb{Z}
_{12}$-module $%
\mathbb{Z}
_{12}$, $S=\left\{  \bar{1},\bar{3},9\right\}  $, $N=\left\langle \bar
{2}\right\rangle $ and $K=\left\langle \bar{3}\right\rangle $. Then $N\cap
K=\left\langle \bar{6}\right\rangle $ is not a weakly prime submodule of $%
\mathbb{Z}
_{12}$.

\begin{proposition}
\label{f}Let $f:M_{1}\rightarrow M_{2}$ be a module homomorphism where $M_{1}$
and $M_{2}$ are two $R$-modules and $S$ be a multiplicatively closed subset of
$R$. Then the following statements hold.
\end{proposition}

\begin{enumerate}
\item If $f$ is an epimorphism and $N$ is a weakly $S$-prime submodule of
$M_{1}$ containing $Ker(f)$, then $f(N)$ is a weakly $S$-prime submodule of
$M_{2}.$

\item If $f$ is a monomorphism and $K$ is a weakly $S$-prime submodule of
$M_{2}$ with $(f^{-1}(K):_{R}M_{1})\cap S=\emptyset$, then $f^{-1}(K)$ is a
weakly $S$-prime submodule of $M_{1}.$
\end{enumerate}

\begin{proof}
(1) First, observe that $(f(N):_{R_{2}}M_{2})\cap S=\emptyset$. Indeed, assume
that $t\in(f(N):_{R_{2}}M_{2})\cap S$. Then $f(tM_{1})=tf(M_{1})=tM_{2}%
\subseteq f(N)$, and so $tM_{1}\subseteq N$ as $K\operatorname{erf}\subseteq
N.$ It follows that $t\in(N:M_{1})\cap S$, a contradiction. Let $s$ be a
weakly $S$-element of $N$ and $a\in R$, $m_{2}\in M_{2}$ with $0\neq am_{2}\in
f(N)$. Then $m_{2}=f(m_{1})$ for some $m_{1}\in M_{1}$ and $0\neq
af(m_{1})=f(am_{1})\in f(N)$ and since $K\operatorname{erf}\subseteq N,$ we
have $0\neq am_{1}\in N$. This yields either $sa\in(N:_{R}M_{1})$ or
$sm_{1}\in N$. Thus, clearly we have either $sa\in(f(N):_{R}M_{2})$ or
$sm_{2}=f(sm_{1})\in f(N)$ as required.

(2) Let $s$ be a weakly $S$-element of $K$ and let $a\in R,$ $m\in M_{1}$ with
$0\neq am\in f^{-1}(K)$. Then $0\neq f(am)=af(m)\in K$ as $f$ is a
monomorphism. Since $K$ is a weakly $S$-prime submodule of $M_{2}$, we have
$sa\in(K:_{R}M_{2})$ or $sf(m)\in K$. Thus, clearly we have $sa\in
(f^{-1}(K):_{R}M_{1})$ or $sm\in f^{-1}(K)$ as needed.
\end{proof}

\begin{corollary}
\label{quot}Let $S$ be a multiplicatively closed subset of a ring $R$ and $N$,
$K$ are two submodules of an $R$-module $M$ with $K\subseteq N$ .
\end{corollary}

\begin{enumerate}
\item If $N$ is a weakly $S$-prime submodule of $M$, then $N/K$ is a weakly
$S$-prime submodule of $M/K$.

\item If $K^{\prime}$ is a weakly $S$-prime submodule of $M$ with $(K^{\prime
}:_{R}N)\cap S=\emptyset$, then $K^{\prime}\cap N$ is a weakly $S$-prime
submodule of $N.$

\item If $N/K$ is a weakly $S$-prime submodule of $M/K$ and $K$ is an
$S$-prime submodule of $M$, then $N$ is an $S$-prime submodule of $M$.

\item If $N/K$ is a weakly $S$-prime submodule of $M/K$ and $K$ is a weakly
$S$-prime submodule of $M$, then $N$ is a weakly $S$-prime submodule of $M$.
\end{enumerate}

\begin{proof}
Note that $(N/K:_{R}M/K)\cap S=\emptyset$ if and only if $(N:_{R}M)\cap
S=\emptyset$.

(1). Consider the canonical epimorphism $\pi:M\rightarrow M/K$ defined by
$\pi(m)=m+K$. Then $\pi(N)=N/K$ is a weakly $S$-prime submodule of $M/K$ by
(1) of Proposition \ref{f}.

(2). Let $K^{\prime}$ be a weakly $S$-prime submodule of $M$ and consider the
natural injection $i:N\rightarrow M$ defined by $i(m)=m$ for all $m\in N$.
Then $(i^{-1}(K^{\prime}):_{R}N)\cap S=\emptyset$. Indeed, if $s\in
(i^{-1}(K^{\prime}):_{R}N)\cap S$, then $sN\subseteq i^{-1}(K^{\prime
})=K^{\prime}\cap N\subseteq K^{\prime}$ and so $s\in(K^{\prime}:_{R}N)\cap
S$, a contradiction. Thus $i^{-1}(K^{\prime})=K^{\prime}\cap N$ is a weakly
$S$-prime submodule of $M$ by (2) of Proposition \ref{f}. $.$

(3). Let $s_{1}$ be a weakly $S$-element of $N/K$ and suppose $K$ is an
$S$-prime submodule of $M$ associated with $s_{2}\in S$. Let $a\in R$ and
$m\in M$ such that $am\in N.$ If $am\in K$, then $s_{2}a\in(K:_{R}%
M)\subseteq(N:_{R}M)$ or $s_{2}m\in K\subseteq N$. If $am\notin K$, then
$K\neq a(m+K)\in N/K$ which implies either $s_{1}a\in(N/K:_{R}M/K)$ or
$s_{1}(m+K)\in N/K$. Thus, $s_{1}a\in(N:_{R}M)$ or $s_{1}m\in N$. It follows
that $N$ is an $S$-prime submodule of $M$ associated with $s=s_{1}s_{2}\in S.$

(4). Similar to (3).
\end{proof}

The next example shows that the converse of Corollary \ref{quot} (1) is not
valid in general.

\begin{example}
Consider the submodules $N=K=\left\langle 6\right\rangle $ of the $%
\mathbb{Z}
$-module $%
\mathbb{Z}
$ and the multiplicatively closed subset $S=\left\{  5^{n}:n\in%
\mathbb{N}
\cup\left\{  0\right\}  \right\}  $ of $%
\mathbb{Z}
.$ It is clear that $N/K$ is a weakly $S$-prime submodule of $%
\mathbb{Z}
/K$ but $N$ is not a weakly $S$-prime submodule of $%
\mathbb{Z}
$ as $0\neq2\cdot3\in N$ but neither $2s\in(N:_{%
\mathbb{Z}
}%
\mathbb{Z}
)$ nor $3s\in N$ for all $s\in S.$
\end{example}

\begin{proposition}
Let $S$ be a multiplicatively closed subset of a ring $R$ and $N$, $K$ be two
weakly $S$-prime submodules of an $R$-module $M$ such that $((N+K):_{R}M)\cap
S=\emptyset$. Then $N+K$ is a weakly $S$-prime submodule of $M.$
\end{proposition}

\begin{proof}
Suppose $N$ and $K$ are two weakly $S$-prime submodules of $M$. By Corollary
\ref{quot} (1), $N/(N\cap K)$ is a weakly $S$-prime submodule of $M/(N\cap
K)$. Now, from the module isomorphism $N/(N\cap K)\cong(N+K)/K,$ we conclude
that $(N+K)/K$ is a weakly $S$-prime submodule of $M/K$. Thus, $N+K$ is a
weakly $S$-prime submodule of $M$ by Corollary \ref{quot} (4).
\end{proof}

\begin{theorem}
\label{cart}Let $S_{1},S_{2}$ be multiplicatively closed subsets of rings
$R_{1}$, $R_{2}$ respectively and $N_{1}$, $N_{2}$ be non-zero submodules of
an $R_{1}$-module $M_{1}$ and an $R_{2}$-module $M_{2}$, respectively.
Consider $M=M_{1}\times M_{2}$ as an $(R_{1}\times R_{2})$-module,
$S=S_{1}\times S_{2}$ and $N=N_{1}\times N_{2}$. Then the following are equivalent.
\end{theorem}

\begin{enumerate}
\item $N$ is a weakly $S$-prime submodule of $M.$

\item $N_{1}$ is an $S_{1}$-prime submodule of $M_{1}$ and $(N_{2}:_{R_{2}%
}M_{2})\cap S_{2}\neq\emptyset$ or $N_{2}$ is an $S_{2}$-prime submodule of
$M_{2}$ and $(N_{1}:_{R_{1}}M_{1})\cap S_{1}\neq\emptyset$

\item $N$ is a $S$-prime submodule of $M.$
\end{enumerate}

\begin{proof}
(1)$\Rightarrow$(2). Suppose $N$ is a weakly $S$-prime submodule of $M$ with
weakly $S$-element $s=(s_{1},s_{2})\in S.$ Assume that $(N_{1}:_{R_{1}}%
M_{1})\cap S_{1}$ and $(N_{2}:_{R_{2}}M_{2})\cap S_{2}$ are both empty. Choose
$0\neq$ $m\in N_{1}.$ Then $(0_{M_{1}},0_{M_{2}})\neq(1,0)(m,1_{M_{2}})\in N$
which implies $(s_{1},s_{2})(1,0)\in(N:_{R}M)=(N_{1}:_{R_{1}}M_{1}%
)\times(N_{2}:_{R_{2}}M_{2})$ or $(s_{1},s_{2})(m,1_{M_{2}})\in N_{1}\times
N_{2}.$ Hence, we have either $s_{1}\in(N_{1}:_{R_{1}}M_{1})\cap S_{1}$ or
$s_{2}\in N_{2}\cap S_{2}\subseteq(N_{2}:_{R_{2}}M_{2})\cap S_{2}$, a
contradiction. Now, we may assume that $(N_{1}:_{R_{1}}M_{1})\cap S_{1}%
\neq\emptyset$ and we show that $N_{2}$ is an $S_{2}$-prime submodule of
$M_{2}.$ Suppose $am^{\prime}\in N_{2}$ for some $a\in R_{2}$ and $m^{\prime
}\in M_{2}$. Then $(0_{M_{1}},0_{M_{2}})\neq(1_{R_{1}},a)(m,m^{\prime})\in N$
implies either $(s_{1},s_{2})(1_{R_{1}},a)\in(N_{1}:_{R_{1}}M_{1})\times
(N_{2}:_{R_{2}}M_{2})$ or $(s_{1},s_{2})(m,m^{\prime})\in N_{1}\times N_{2}$.
Thus, $s_{2}a\in(N_{2}:_{R_{2}}M_{2})$ or $s_{2}m^{\prime}\in N_{2}$ and so
$N_{2}$ is an $S_{2}$-prime submodule of $M_{2}.$

(2)$\Rightarrow$(3). Follows from \cite[Theorem 2.14]{S-prime subm}

(3)$\Rightarrow$(1). Straightforward.
\end{proof}

\begin{theorem}
\label{cart2}Let $M=M_{1}\times M_{2}\times\cdots\times M_{n}$ be an
$R_{1}\times R_{2}\times\cdots\times R_{n}$-module and $S=S_{1}\times
S_{2}\times\cdots\times S_{n}$ where $R_{i}$'s are rings, $S_{i}$ is a
multiplicatively closed subset of $R_{i}$ and $N_{i}$ is a non-zero submodule
of $M_{i}$ for each $i=1,2,...,n$. Then the following assertions are equivalent.
\end{theorem}

\begin{enumerate}
\item $N=N_{1}\times N_{2}\times\cdots\times N_{n}$ is a weakly $S$-prime
submodule of $M.$

\item For $i=1,2,...,n$, $N_{i}$ is an $S$-prime submodule of $M_{i}$ and
$(N_{j}:_{R_{j}}M_{j})\cap S_{j}\neq\emptyset$ for all $j\neq i.$
\end{enumerate}

\begin{proof}
We prove the claim by using mathematical induction on $n$. The claim follows
by Theorem \ref{cart} for $n=2.$ Now, we assume that the claim holds for all
$k<n$ and prove it for $k=n$. Suppose $N=N_{1}\times N_{2}\times\cdots\times
N_{n}$ is a weakly $S$-prime submodule of $M$. Then Theorem \ref{cart} implies
that $N=N^{\prime}\times N_{n}$ where, say, $N^{\prime}=N_{1}\times
N_{2}\times\cdots\times N_{n-1}$ is a weakly $S$-prime submodule of
$M^{\prime}=M_{1}\times M_{2}\times\cdots\times M_{n-1}$ and $S_{n}\cap
(N_{n}:_{R_{n}}M_{n})\neq\emptyset$. Thus, the result follows by the induction hypothesis.
\end{proof}

Let $M$ be an $R$-module and $S$ be a multiplicatively closed subset of $R$
with $S\cap Ann_{R}(M)=\phi$. Following \cite{S-prime subm}, $M$ is called
$S$-torsion-free if there is $s\in S$ such that whenever $rm=0$ for $r\in R$
and $m\in M$, then $sr=0$ or $sm=0$. Compare with \cite[Proposition
2.24]{S-prime subm}, we have the following result.

\begin{proposition}
Let $S$ be a multiplicatively closed subset of a ring $R$ and $N$ be a
submodule of an $S$-torsion-free $R$-module $M$. If $\eta:R\rightarrow
R/(N:_{R}M)$ is the canonical homomorphism, then $N$ is weakly $S$-prime in
$M$ if and only if $M/N$ is a $\eta(S)$-torsion-free $R/(N:_{R}M)$-module.
\end{proposition}

\begin{proof}
First, we clearly note that $s\in S\cap(N:_{R}M)$ if and only if $\bar{s}%
\in\eta(S)\cap Ann_{R/(N:_{R}M)}(M/N)$.

$\Rightarrow)$ Suppose $N$ is weakly $S$-prime in $M$ with weakly $S$-element
$s_{1}\in S$. Let $\bar{r}\in R/(N:_{R}M)$, $\bar{m}\in M/N$ such that
$\bar{r}\bar{m}=\bar{0}$. Then $rm\in N$ and we have two cases. If $rm=0$,
then by assumption there is $s_{2}\in S$ such that $s_{2}r=0$ or $s_{2}m=0$.
Thus $\bar{s}_{2}\bar{r}=\bar{0}$ or $\bar{s}_{2}\bar{m}=\bar{0}$ where
$\bar{s}_{2}\in\eta(S)$ as needed. If $rm\neq0$, then $s_{1}r\in(N:_{R}M)$ or
$s_{1}m\in N$ and so $\bar{s}_{1}\bar{r}=\bar{0}_{R/(N:_{R}M)}$ or $\bar
{s}_{1}\bar{m}=\bar{0}_{M/N}$ where $\bar{s}_{1}\in\eta(S)$. Therefore, $M/N$
is a $\eta(S)$-torsion-free $R/(N:_{R}M)$-module associated to $\bar{s}%
_{1}\bar{s}_{2}\in\eta(S)$.

$\Leftarrow)$ Follows directly by \cite[Proposition 2.24]{S-prime subm}.
\end{proof}

Let $R$ be a ring and $M$ be an $R$-module. Recall that the idealization of
$M$ in $R$ denoted by $R\ltimes M$ is the commutative ring $R\oplus M$ with
coordinate-wise addition and multiplication defined as $(r_{1},m_{1}%
)(r_{2},m_{2})=(r_{1}r_{2},r_{1}m_{2}+r_{2}m_{1})$ \cite{Nagata}. For an ideal
$I$ of $R$ and a submodule $N$ of $M,$ The set $I\ltimes N=I\oplus N\ $is not
always an ideal of $R\ltimes M\ $and it is an ideal if and only if
$IM\subseteq N$ \cite[Theorem 3.1]{AnWi}$.$ Among many other properties of an
ideal $I\ltimes N$ of $R\ltimes M$, we have $\sqrt{I\ltimes N}=\sqrt{I}\ltimes
M$ and in particular, $\sqrt{0\ltimes0}=\sqrt{0}\ltimes M$, \cite[Theorem
3.2]{AnWi}.$\ $It is clear that if $S$ is a multiplicatively closed subset of
$R$ and $N$ a submodule of $M$, then $S\ltimes K=\{(s,k):$ $s\in S$, $k\in
K\}$ is a multiplicatively closed subset of $R\ltimes M$. In \cite[Proposition
27]{WS-prime}, it is proved that if $I\ltimes M$ is a weakly $S\ltimes
M$-prime (or weakly $S\ltimes0$-prime ideal of $R\ltimes M$ where $I$ is an
ideal of $R$ disjoint with $S$, then $I$ is a weakly $S$-prime ideal of $R$.
In general, we have:

\begin{theorem}
\label{Ideal}Let $S$ be a multiplicatively closed subset of a ring $R,$ $I$ be
an ideal of $R$ and $K\subseteq N$ be submodules of an $R$-module $M$ with
$IM\subseteq N$. Let $I\ltimes N$ be a weakly $S\ltimes K$-prime ideal of
$R\ltimes M$. Then
\end{theorem}

\begin{enumerate}
\item $I$ is a weakly $S$-prime ideal of $R$ and $N$ a weakly $S$-prime
submodule of $M$ whenever $(N:_{R}M)\cap S=\phi$.

\item There exists $s\in S$ such that for all $a,b\in R$, $ab=0$, $sa\notin
I$, $sb\notin I$ implies $a,b\in ann(N)$ and for all $c\in R$, $m\in M$,
$cm=0$, $sc\notin(N:_{R}M)$, $sm\notin N$ implies $c\in ann(I)$ and
$m\in(0:_{M}I)$.

\item If $I\ltimes N$ is not $S\ltimes K$-prime, then $(s,k)(I\ltimes
N)=(sI\ltimes0)\oplus(0\ltimes sN+Ik)$ for some $(s,k)\in S\ltimes K.$

\item If $I\ltimes N$ is $S\ltimes K$-prime then $sM\subseteq N$ for some
$s\in S$.
\end{enumerate}

\begin{proof}
Let $(s,k)\in S\ltimes K$ be a weakly $S\ltimes K$-element of $I\ltimes N$.

(1) Note that clearly $(S\ltimes K)\cap(I\ltimes N)=\phi$ if and only if
$I\cap S=\phi$. Suppose that $a,b\in R$ with $0\neq ab\in I$. Then
$(0,0)\neq(a,0)(b,0)\in I\ltimes N$ implies that either $(s,k)(a,0)\in
I\ltimes N$ or $(s,k)(b,0)\in I\ltimes N$. Thus, either $sa\in I$ or $sb\in I$
and $I$ is weakly $S$-prime in $R$. Now, let $0\neq rm\in N$ for $r\in R$,
$m\in M$. Then $(0,0)\neq(r,0)(0,m)\in I\ltimes N$ and so
$(sr,rk)=(s,k)(r,0)\in I\ltimes N$ or $(0,sm)=(s,k)(0,m)\in I\ltimes N$. In
the first case, we get $sr\in I\subseteq(N:_{R}M)$ and the second case implies
$sm\in N$. Therefore, $N$ is a weakly $S$-prime submodule of $M$.

(2) Let $a,b\in R$ such that $ab=0$ and $sa\notin I$, $sb\notin I$. Suppose
$a\notin ann(N)$ so that there exists $n\in N$ such that $an\neq0$. Thus,
$(0,0)\neq(a,0)(b,n)=(0,an)\in I\ltimes N$ and so either $(s,k)(a,0)\in
I\ltimes N$ or $(s,k)(b,n)\in I\ltimes N$. Hence, $sa\in I$ or $sb\in I$, a
contradiction. Similarly, if $b\notin ann(N),$ then we get a contradiction.
Therefore, $a,b\in ann(N)$ as needed. Next, we assume $cm=0$ for $c\in R$,
$m\in M$ and $sc\notin(N:_{R}M)$, $sm\notin N$. We have two cases.

Case 1. If $c\notin ann(I)$, then there exists $a\in I$ such that $ca\neq0$.
Hence, $(0,0)\neq(c,0)(a,m)=(ca,0)\in I\ltimes N$ and so $(s,k)(c,0)\in
I\ltimes N$ or $(s,k)(a,m)\in I\ltimes N$. Therefore, $sc\in I\subseteq
(N:_{R}M)$ or $sm+ka\in N$ (and so $sm\in N$ as $K\subseteq N$) which
contradicts the assumption.

Case 2. If $m\notin(0:_{M}I)$, then there exists $a\in I$ such that $am\neq0$.
Thus, $(0,0)\neq(a,m)(c,m)=(ac,am)\in I\ltimes N$ implies either
$(s,k)(a,m)\in I\ltimes N$ or $(s,k)(c,m)\in I\ltimes N$. It follows that
either $sc\in I\subseteq(N:_{R}M)$ or $sm\in N$ which is also a contradiction.

(3) If $I\ltimes N$ is not $S\ltimes K$-prime, then $(s,k)(I\ltimes
N)(\sqrt{0}\ltimes M)=(0,0)$ for some $(s,k)\in S\ltimes K$ by
\cite[Proposition 9]{WS-prime}. Thus, by \cite[Theorem 3.3]{AnWi} $s\sqrt
{0}I\ltimes(sIM+s\sqrt{0}N+\sqrt{0}Ik)=(0,0)$. Then clearly $sIM=0$ and so
$sI\ltimes0$ is an ideal of $R\ltimes M$. Now, $(s,k)(I\ltimes N)=sI\ltimes
(sN+Ik)=(sI\ltimes0)\oplus(0\ltimes sN+Ik)$ as required.

(4) If $I\ltimes N$ is $S\ltimes K$-prime in $R\ltimes M$, then $(s,k)(\sqrt
{0}\ltimes M)\subseteq(I\ltimes N)$ for some $(s,k)\in S\ltimes K$ by
\cite[Corollary 6]{WS-prime}. Thus, $s\sqrt{0}\ltimes(sM+\sqrt{0}%
k)\subseteq(I\ltimes N)$ and so clearly, $sM\subseteq N$ as needed.
\end{proof}

In general if $I$ is a (weakly) $S$-prime ideal of a ring $R$ and $N$ a
(weakly) $S$-prime submodule of an $R$-module $M$, then $I\ltimes N$ need not
be a (weakly) $S\ltimes K$-prime ideal of $R\ltimes M$.

\begin{example}
Consider the multiplicatively closed subset $S=\left\{  3^{n}:n\in%
\mathbb{N}
\right\}  $ of $%
\mathbb{Z}
$. While clearly $0$ is (weakly) $S$-prime in $%
\mathbb{Z}
$ and $\left\langle \bar{2}\right\rangle $ is (weakly) $S$-prime in the $%
\mathbb{Z}
$-module $%
\mathbb{Z}
_{6}$, the ideal $0\ltimes\left\langle \bar{2}\right\rangle $ is not (weakly)
$S\ltimes0$-prime in $%
\mathbb{Z}
\ltimes%
\mathbb{Z}
_{6}$. Indeed, $(0,0)\neq(0,\bar{1})(2,\bar{1})=(0,\bar{2})\in0\ltimes
\left\langle \bar{2}\right\rangle $ but $(s,\bar{0})(0,\bar{1})\notin
0\ltimes\left\langle \bar{2}\right\rangle $ and $(s,\bar{0})(2,\bar{1}%
)\notin0\ltimes\left\langle \bar{2}\right\rangle $ for all $s\in S$.
\end{example}

\section{(Weakly) S-prime Submodules of Amalgamation Modules}

Let $R$ be a ring, $J$ an ideal of $R$ and $M$ an $R$-module. We recall that
the set%

\[
R\Join J=\left\{  (r,r+j):r\in R\text{ , }j\in J\right\}
\]
is a subring of $R\times R$ called the the amalgamated duplication of $R$
along $J$, see \cite{Danna}. Recently, in \cite{Bouba}, the duplication of the
$R$-module $M$ along the ideal $J$ denoted by $M\Join J$ is defined as%

\[
M\Join J=\left\{  (m,m^{\prime})\in M\times M:m-m^{\prime}\in JM\right\}
\]
which is an $(R\Join J)$-module with scalar multiplication defined by
$(r,r+j).(m,m^{\prime})=(rm,(r+j)m^{\prime})$ for $r\in R$, $j\in J$ and
$(m,m^{\prime})\in M\Join J$. Many properties and results concerning this kind
of modules can be found in \cite{Bouba}.

Let $N$ be a submodule of an $R$-module $M$ and $J$ be an ideal of $R$. Then clearly%

\[
N\Join J=\left\{  (n,m)\in N\times M:n-m\in JM\right\}
\]
and
\[
\bar{N}=\left\{  (m,n)\in M\times N:m-n\in JM\right\}
\]

are submodules of $M\Join J$. If $S$ is a multiplicatively closed subset of
$R$, then obviously, the sets $S\Join J=\left\{  (s,s+j):s\in S\text{, }j\in
J\right\}  $ and $\bar{S}=\left\{  (r,r+j):r+j\in S\right\}  $ are
multiplicatively closed subsets of $R\Join J$.

In general, let $f:R_{1}\rightarrow R_{2}$ be a ring homomorphism, $J$ be an
ideal of $R_{2}$, $M_{1}$ be an $R_{1}$-module, $M_{2}$ be an $R_{2}$-module
(which is an $R_{1}$-module induced naturally by $f$) and $\varphi
:M_{1}\rightarrow M_{2}$ be an $R_{1}$-module homomorphism. The subring
\[
R_{1}\Join^{f}J=\left\{  (r,f(r)+j):r\in R_{1}\text{, }j\in J\right\}
\]
of $R_{1}\times R_{2}$ is called the amalgamation of $R_{1}$ and $R_{2}$ along
$J$ with respect to $f$. In \cite{Rachida}, the amalgamation of $M_{1}$ and
$M_{2}$ along $J$ with respect to $\varphi$ is defined as%

\[
M_{1}\Join^{\varphi}JM_{2}=\left\{  (m_{1},\varphi(m_{1})+m_{2}):m_{1}\in
M_{1}\text{ and }m_{2}\in JM_{2}\right\}
\]
which is an $(R_{1}\Join^{f}J)$-module with the scalar product defined as
\[
(r,f(r)+j)(m_{1},\varphi(m_{1})+m_{2})=(rm_{1},\varphi(rm_{1})+f(r)m_{2}%
+j\varphi(m_{1})+jm_{2})
\]
For submodules $N_{1}$ and $N_{2}$ of $M_{1}$ and $M_{2}$, respectively,
clearly the sets
\[
N_{1}\Join^{\varphi}JM_{2}=\left\{  (m_{1},\varphi(m_{1})+m_{2})\in M_{1}%
\Join^{\varphi}JM_{2}:m_{1}\in N_{1}\right\}
\]
and
\[
\overline{N_{2}}^{\varphi}=\left\{  (m_{1},\varphi(m_{1})+m_{2})\in M_{1}%
\Join^{\varphi}JM_{2}:\text{ }\varphi(m_{1})+m_{2}\in N_{2}\right\}
\]
are submodules of $M_{1}\Join^{\varphi}JM_{2}$. Moreover if $S_{1}$ and
$S_{2}$ are multiplicatively closed subsets of $R_{1}$ and $R_{2}$,
respectively, then
\[
S_{1}\Join^{f}J=\left\{  (s_{1},f(s_{1})+j):s\in S_{1}\text{, \ }j\in
J\right\}
\]
and
\[
\overline{S_{2}}^{\varphi}=\left\{  (r,f(r)+j):r\in R_{1}\text{, \ }f(r)+j\in
S_{2}\right\}
\]
are clearly multiplicatively closed subsets of $M_{1}\Join^{\varphi}JM_{2}$.

Note that if $R=R_{1}=R_{2}$, $M=M_{1}=M_{2}$, $f=Id_{R}$ and $\varphi=Id_{M}%
$, then the amalgamation of $M_{1}$ and $M_{2}$ along $J$ with respect to
$\varphi$ is exactly the duplication of the $R$-module $M$ along the ideal
$J$. Moreover, in this case, we have $N_{1}\Join^{\varphi}JM_{2}=N\Join J$,
$\overline{N_{2}}^{\varphi}=\bar{N}$, $S_{1}\Join^{f}J=S\Join J$ and
$\overline{S_{2}}^{\varphi}=\bar{S}$.

\begin{theorem}
\label{Amalg}Consider the $(R_{1}\Join^{f}J)$-module $M_{1}\Join^{\varphi
}JM_{2}$ defined as above. Let $S$ be a multiplicatively closed subsets of
$R_{1}$ and $N_{1}$ be submodule of $M_{1}$. Then
\end{theorem}

\begin{enumerate}
\item $N_{1}\Join^{\varphi}JM_{2}$ is an $S\Join^{f}J$-prime submodule of
$M_{1}\Join^{\varphi}JM_{2}$ if and only if $N_{1}$ is an $S$-prime submodule
of $M_{1}$.

\item $N_{1}\Join^{\varphi}JM_{2}$ is a weakly $S\Join^{f}J$-prime submodule
of $M_{1}\Join^{\varphi}JM_{2}$ if and only if $N_{1}$ is a weakly $S$-prime
submodule of $M_{1}$ and for $r_{1}\in R_{1}$, $m_{1}\in M_{1}$ with
$r_{1}m_{1}=0$ but $s_{1}r_{1}\notin(N_{1}:_{R_{1}}M_{1})$ and $s_{1}%
m_{1}\notin N_{1}$ for all $s_{1}\in S$, then $f(r_{1})m_{2}+j\phi
(m_{1})+jm_{2}=0$ for every $j\in J$ and $m_{2}\in JM_{2}$.
\end{enumerate}

\begin{proof}
We clearly note that $(N_{1}\Join^{\varphi}JM_{2}:_{R_{1}\Join^{f}J}M_{1}%
\Join^{\varphi}JM_{2})\cap S\Join^{f}J=\phi$ if and only if $(N_{1}:_{R_{1}%
}M_{1})\cap S=\phi$.

(1) Suppose $(s,f(s)+j)$ is an $S\Join^{f}J$-element of $N_{1}\Join^{\varphi
}JM_{2}$ and let $r_{1}m_{1}\in N_{1}$ for $r_{1}\in R_{1}$ and $m_{1}\in
M_{1}$. Then $(r_{1},f(r_{1}))\in R_{1}\Join^{f}J$ and $(m_{1},\varphi
(m_{1}))\in M_{1}\Join^{\varphi}JM_{2}$ with $(r_{1},f(r_{1}))(m_{1}%
,\varphi(m_{1}))=(r_{1}m_{1},\varphi(r_{1}m_{1}))\in N_{1}\Join^{\varphi
}JM_{2}$. Thus, either
\[
(s,f(s)+j)(r_{1},f(r_{1}))\in(N_{1}\Join^{\varphi}JM_{2}:_{R_{1}\Join^{f}%
J}M_{1}\Join^{\varphi}JM_{2})
\]
or%
\[
(s,f(s)+j)(m_{1},\varphi(m_{1}))\in N_{1}\Join^{\varphi}JM_{2}%
\]
In the first case, for all $m\in M_{1}$, $(s,f(s)+j)(r_{1},f(r_{1}%
))(m,\varphi(m))\in N_{1}\Join^{\varphi}JM_{2}$ and so $sr_{1}M_{1}\subseteq
N_{1}$. In the second case, $sm_{1}\in N_{1}$ and so $N_{1}$ is an $S$-prime
submodule of $M_{1}$. Conversely, let $s$ be an $S$-element of $N_{1}$. Let
$(r_{1},f(r_{1})+j_{1})\in R_{1}\Join^{f}J$ and $(m_{1},\varphi(m_{1}%
)+m_{2})\in M_{1}\Join^{\varphi}JM_{2}$ such that
\begin{align*}
&  (r_{1}m_{1},\varphi(r_{1}m_{1})+f(r_{1})m_{2}+j_{1}\varphi(m_{1}%
)+j_{1}m_{2})\\
&  =(r_{1},f(r_{1})+j_{1})(m_{1},\varphi(m_{1})+m_{2})\in N_{1}\Join^{\varphi
}JM_{2}%
\end{align*}
Then $r_{1}m_{1}\in N_{1}$ and hence either $sr_{1}M_{1}\subseteq N_{1}$ or
$sm_{1}\in N_{1}$. If $sr_{1}M_{1}\subseteq N_{1}$, then clearly
$(s,f(s))(r_{1},f(r_{1})+j_{1})\in(N_{1}\Join^{\varphi}JM_{2}:_{R_{1}\Join
^{f}J}M_{1}\Join^{\varphi}JM_{2})$ and if $sm_{1}\in N_{1}$, then
$(s,f(s))(m_{1},\varphi(m_{1})+m_{2})\in N_{1}\Join^{\varphi}JM_{2}$.
Therefore, $N_{1}\Join^{\varphi}JM_{2}$ is an $S\Join^{f}J$-prime submodule of
$M_{1}\Join^{\varphi}JM_{2}$ associated to $(s,f(s))\in S\Join^{f}J$.

(2) Suppose $(s,f(s)+j)$ is a weakly $S\Join^{f}J$-element of $N_{1}%
\Join^{\varphi}JM_{2}$. Let $r_{1}\in R_{1}$ and $m_{1}\in M_{1}$ such that
$0\neq r_{1}m_{1}\in N_{1}$ so that $(0,0)\neq(r_{1},f(r_{1}))(m_{1}%
,\varphi(m_{1}))=(r_{1}m_{1},\varphi(r_{1}m_{1}))\in N_{1}\Join^{\varphi
}JM_{2}$. By assumption, either $(s,f(s)+j)(r_{1},f(r_{1}))\in(N_{1}%
\Join^{\varphi}JM_{2}:_{R_{1}\Join^{f}J}M_{1}\Join^{\varphi}JM_{2})$ or
$(s,f(s)+j)(m_{1},\varphi(m_{1}))\in N_{1}\Join^{\varphi}JM_{2}$ and so
$N_{1}$ is $S$-prime in $M_{1}$ as in the proof of (1). Now, we use the
contrapositive to prove the other part. Let $r_{1}\in R_{1}$, $m_{1}\in M_{1}$
with $r_{1}m_{1}=0$ and $f(r_{1})m_{2}+j\phi(m_{1})+jm_{2}\neq0$ for some
$j\in J$ and some $m_{2}\in JM_{2}$. Then
\begin{align*}
(0,0)  &  \neq(r_{1},f(r_{1})+j)(m_{1},\varphi(m_{1})+m_{2})\\
&  =(0,f(r_{1})m_{2}+j\varphi(m_{1})+jm_{2})\in N_{1}\Join^{\varphi}JM_{2}%
\end{align*}
By assumption, either $(s,f(s)+j)(r_{1},f(r_{1})+j)\in(N_{1}\Join^{\varphi
}JM_{2}:_{R_{1}\Join^{f}J}M_{1}\Join^{\varphi}JM_{2})$ or $(s,f(s)+j)(m_{1}%
,\varphi(m_{1})+m_{2})\in N_{1}\Join^{\varphi}JM_{2}$ and so again $sr_{1}%
\in(N_{1}:_{R_{1}}M_{1})$ or $sm_{1}\in N_{1}$ as needed. Conversely, let $s$
be a weakly $S$-element of $N_{1}$ and let $(r_{1},f(r_{1})+j)\in R_{1}%
\Join^{f}J$ and $(m_{1},\varphi(m_{1})+m_{2})\in M_{1}\Join^{\varphi}JM_{2}$
such that
\begin{align*}
(0,0)  &  \neq(r_{1}m_{1},\varphi(r_{1}m_{1})+f(r_{1})m_{2}+j\varphi
(m_{1})+jm_{2})\\
&  =(r_{1},f(r_{1})+j)(m_{1},\varphi(m_{1})+m_{2})\in N_{1}\Join^{\varphi
}JM_{2}%
\end{align*}
If $0\neq r_{1}m_{1}$, then the proof is similar to that of (1). Suppose
$r_{1}m_{1}=0$. Then $f(r_{1})m_{2}+j\varphi(m_{1})+jm_{2}\neq0$ and so by
assumption there exists $s^{\prime}\in S$ such that either $s^{\prime}r_{1}%
\in(N_{1}:_{R_{1}}M_{1})$ or $s^{\prime}m_{1}\in N_{1}$. Thus, $(s^{\prime
},f(s^{\prime}))(r_{1},f(r_{1})+j)\in(N_{1}\Join^{\varphi}JM_{2}:_{R_{1}%
\Join^{f}J}M_{1}\Join^{\varphi}JM_{2})$ or $(s^{\prime},f(s^{\prime}%
))(m_{1},\varphi(m_{1})+m_{2})\in N_{1}\Join^{\varphi}JM_{2}$. Therefore,
$N_{1}\Join^{\varphi}JM_{2}$ is a weakly $S\Join^{f}J$-prime submodule of
$M_{1}\Join^{\varphi}JM_{2}$ associated to $(ss^{\prime},f(ss^{\prime}))\in
S\Join^{f}J$.
\end{proof}

In particular, if $S$ is a multiplicatively closed subset of $R_{1}$, then
$S\times f(S)$ is a multiplicatively closed subset of $R_{1}\Join^{f}J$.
Moreover, one can similarly prove Theorem \ref{Amalg} if we consider $S\times
f(S)$ instead of $S\Join^{f}J$.

\begin{corollary}
\label{ca1}Consider the $(R_{1}\Join^{f}J)$-module $M_{1}\Join^{\varphi}%
JM_{2}$ defined as in Theorem \ref{Amalg} and let $N_{1}$ be a submodule of
$M_{1}$. Then
\end{corollary}

\begin{enumerate}
\item $N_{1}\Join^{\varphi}JM_{2}$ is a prime submodule of $M_{1}%
\Join^{\varphi}JM_{2}$ if and only if $N_{1}$ is a prime submodule of $M_{1}$.

\item $N_{1}\Join^{\varphi}JM_{2}$ is a weakly prime submodule of $M_{1}%
\Join^{\varphi}JM_{2}$ if and only if $N_{1}$ is a weakly prime submodule of
$M_{1}$ and for $r_{1}\in R_{1}$, $m_{1}\in M_{1}$ with $r_{1}m_{1}=0$ but
$r_{1}\notin(N_{1}:_{R_{1}}M_{1})$ and $m_{1}\notin N_{1}$, then
$f(r_{1})m_{2}+j\phi(m_{1})+jm_{2}=0$ for every $j\in J$ and $m_{2}\in JM_{2}$.
\end{enumerate}

\begin{proof}
We just take $S=\left\{  1_{R_{1}}\right\}  $ (and so $S\times f(S)=\left\{
(1_{R_{1}},1_{R_{2}})\right\}  $) and use Theorem \ref{Amalg}.
\end{proof}

\begin{theorem}
\label{Amalg2}Consider the $(R_{1}\Join^{f}J)$-module $M_{1}\Join^{\varphi
}JM_{2}$ defined as in Theorem \ref{Amalg} where $f$ and $\varphi$ are
epimorphisms. Let $S$ be a multiplicatively closed subsets of $R_{2}$ and
$N_{2}$ be a submodule of $M_{2}$. Then
\end{theorem}

\begin{enumerate}
\item $N_{2}$ is an $S$-prime submodule of $M_{2}$ if and only if
$\overline{N_{2}}^{\varphi}$ is an $\overline{S}^{\varphi}$-prime submodule of
$M_{1}\Join^{\varphi}JM_{2}$.

\item If $\overline{N_{2}}^{\varphi}$ is an $\overline{S}^{\varphi}$-prime
submodule of $M_{1}\Join^{\varphi}JM_{2}$, and $(N_{2}:_{R_{2}}JM_{2})\cap
S=\phi$, then $(N_{2}:_{M_{2}}J)$ is an $S$-prime submodule of $M_{2}$.
\end{enumerate}

\begin{proof}
(1). We note that $(\overline{N_{2}}^{\varphi}:_{R_{1}\Join^{f}J}M_{1}%
\Join^{\varphi}JM_{2})\cap\overline{S}^{\varphi}=\phi$ if and only if
$(N_{2}:_{R_{2}}M_{2})\cap S=\phi$. Indeed if $(t,f(t)+j)=(t,s)\in\overline
{S}^{\varphi}$ such that $(t,s)(M_{1}\Join^{\varphi}JM_{2})\subseteq
\overline{N_{2}}^{\varphi}$, then for each $m_{2}=\varphi(m_{1})\in M_{2}$, we
have $(t,s)(m_{1},m_{2})\in\overline{N_{2}}^{\varphi}$. Therefore, $sm_{2}\in
N_{2}$ and $s\in(N_{2}:_{R_{2}}M_{2})$. The converse is similar.

Suppose $N_{2}$ is an $S$-prime submodule of $M_{2}$ associated to $s=f(t)\in
S$. Let $(r_{1},f(r_{1})+j)\in R_{1}\Join^{f}J$ and $(m_{1},\varphi
(m_{1})+m_{2})\in M_{1}\Join JM_{2}$ such that
\[
(r_{1},f(r_{1})+j)(m_{1},\varphi(m_{1})+m_{2})\in\overline{N_{2}}^{\varphi}%
\]

Then $(f(r_{1})+j)(\varphi(m_{1})+m_{2})\in N_{2}$ and so $s(f(r_{1}%
)+j)\in(N_{2}:_{R_{2}}M_{2})$ or $s(\varphi(m_{1})+m_{2})\in N_{2}$. If
$s(f(r_{1})+j)\in(N_{2}:_{R_{2}}M_{2})$, then for all $(m_{1},\varphi
(m_{1})+m_{2})\in M_{1}\Join^{\varphi}JM_{2}$, clearly $(t,s)(r_{1}%
,f(r_{1})+j)(m_{1},\varphi(m_{1})+m_{2})\in\overline{N_{2}}^{\varphi}$ and so
$(t,s)(r_{1},f(r_{1})+j)\in(\overline{N_{2}}^{\varphi}:_{R_{1}\Join^{f}J}%
M_{1}\Join^{\varphi}JM_{2})$. If $s(\varphi(m_{1})+m_{2})\in N_{2}$, then
$(t,s)(m_{1},\varphi(m_{1})+m_{2})\in\overline{N_{2}}^{\varphi}$ and the
result follows. Conversely, suppose $\overline{N_{2}}^{\varphi}$ is an
$\overline{S}^{\varphi}$-prime submodule of $M_{1}\Join^{\varphi}JM_{2}$
associated to $(t,f(t)+j)=(t,s)\in\overline{S}^{\varphi}$. Let $r_{2}%
=f(r_{1})\in R_{2}$ and $m_{2}=\varphi(m_{1})\in M_{2}$ such that $r_{2}%
m_{2}\in N_{2}$. Then $(r_{1},r_{2})\in R_{1}\Join^{f}J$ and $(m_{1},m_{2})\in
M_{1}\Join^{\varphi}JM_{2}$ with $(r_{1},r_{2})(m_{1},m_{2})\in\overline
{N_{2}}^{\varphi}$. Thus, $(t,s)(r_{1},r_{2})(M_{1}\Join^{\varphi}%
JM_{2})\subseteq\overline{N_{2}}^{\varphi}$ or $(t,s)(m_{1},m_{2})\in
\overline{N_{2}}^{\varphi}$. If $(t,s)(r_{1},r_{2})(M_{1}\Join^{\varphi}%
JM_{2})\subseteq\overline{N_{2}}^{\varphi}$, then for all $m=\varphi
(m^{\prime})\in M_{2}$, we have $(t,s)(r_{1},r_{2})(m^{\prime},m)\in
\overline{N_{2}}^{\varphi}$ and so $sr_{2}M_{2}\subseteq N_{2}$. If
$(t,s)(m_{1},m_{2})\in\overline{N_{2}}^{\varphi}$, then $sm_{2}\in N_{2}$ and
we are done.

(2). Suppose $\overline{N_{2}}^{\varphi}$ is an $\overline{S}^{\varphi}$-prime
submodule of $M_{1}\Join^{\varphi}JM_{2}$ associated to $(t,f(t)+j^{\prime
})=(t,s)\in$ $\overline{S}^{\varphi}$. Let $r_{2}\in R_{2}$, $m_{2}\in M_{2}$
such that $r_{2}m_{2}\in(N_{2}:_{M_{2}}J)$. Then $r_{2}Jm_{2}\subseteq N_{2}$
and so for all $j\in J$, we have $(r_{1},f(r_{1}))(0,jm_{2})\in\overline
{N_{2}}^{\varphi}$ where $f(r_{1})=r_{2}$. By assumption, $(t,s)(r_{1}%
,r_{2})\in(\overline{N_{2}}^{\varphi}:_{R_{1}\Join^{f}J}M_{1}\Join^{\varphi
}JM_{2})$ or $(t,s)(0,jm_{2})\in\overline{N_{2}}^{\varphi}$. If $(t,s)(r_{1}%
,r_{2})\in(\overline{N_{2}}^{\varphi}:_{R_{1}\Join^{f}J}M_{1}\Join^{\varphi
}JM_{2})$, then for all $m_{2}\in M_{2}$ and all $j\in J$, we have
$(t,s)(r_{1},r_{2})(0,jm_{2})\in\overline{N_{2}}^{\varphi}$ and so
$sr_{2}jm_{2}\in N_{2}$. Thus, $sr_{2}\in(N_{2}:_{R_{2}}JM_{2})=((N_{2}%
:_{M_{2}}J):_{R_{2}}M_{2})$. If $(t,s)(r_{1},r_{2})\notin(\overline{N_{2}%
}^{\varphi}:_{R_{1}\Join^{f}J}M_{1}\Join^{\varphi}JM_{2})$, then
$(t,s)(0,jm_{2})\in\overline{N_{2}}^{\varphi}$ for all $j\in J$ and so
$sm_{2}\in(N_{2}:_{M_{2}}J)$ as required.
\end{proof}

In particular, if we consider $S=\left\{  1_{R_{2}}\right\}  $ and take
$T=\left\{  (1_{R_{1}},1_{R_{2}})\right\}  $ instead of $\overline{S}%
^{\varphi}$ in Theorem \ref{Amalg2}, then we get the following corollary.

\begin{corollary}
Consider the $(R_{1}\Join^{f}J)$-module $M_{1}\Join^{\varphi}JM_{2}$ defined
as in Theorem \ref{Amalg} where $f$ and $\varphi$ are epimorphisms and let
$N_{2}$ be a submodule of $M_{2}$. Then
\end{corollary}

\begin{enumerate}
\item $N_{2}$ is a prime submodule of $M_{2}$ if and only if $\overline{N_{2}%
}^{\varphi}$ is a prime submodule of $M_{1}\Join^{\varphi}JM_{2}$.

\item If $\overline{N_{2}}^{\varphi}$ is a prime submodule of $M_{1}%
\Join^{\varphi}JM_{2}$ and $J\nsubseteq(N_{2}:_{R_{2}}M_{2})$, then
$(N_{2}:_{M_{2}}J)$ is a prime submodule of $M_{2}$.
\end{enumerate}

\begin{theorem}
\label{Amalg3}Consider the $(R_{1}\Join^{f}J)$-module $M_{1}\Join^{\varphi
}JM_{2}$ defined as in Theorem \ref{Amalg} where $f$ and $\varphi$ are
epimorphisms. Let $S$ be a multiplicatively closed subsets of $R_{2}$ and
$N_{2}$ be a submodule of $M_{2}$. Then
\end{theorem}

\begin{enumerate}
\item $\overline{N_{2}}^{\varphi}$ is a weakly $\overline{S}^{\varphi}$-prime
submodule of $M_{1}\Join^{\varphi}JM_{2}$ if and only if $N_{2}$ is a weakly
$S$-prime submodule of $M_{2}$ and for $r_{1}\in R_{1}$, $m_{1}\in M_{1}$,
$m_{2}\in JM_{2}$, $j\in J$ with $(f(r_{1})+j)(\varphi(m_{1})+m_{2})=0$ but
$s(f(r_{1})+j)\notin(N_{2}:_{R_{2}}M_{2})$ and $s(\varphi(m_{1})+m_{2})\notin
N_{2}$ for all $s\in S$, then $r_{1}m_{1}=0$.

\item If $\overline{N_{2}}^{\varphi}$ is a weakly $\overline{S}^{\varphi}%
$-prime submodule of $M_{1}\Join^{\varphi}JM_{2}$, $(N_{2}:_{R_{2}}JM_{2})\cap
S=\phi$ and $Z_{R_{2}}(M_{2})\cap J=\left\{  0\right\}  $, then $(N_{2}%
:_{M_{2}}J)$ is a weakly $S$-prime submodule of $M_{2}$.
\end{enumerate}

\begin{proof}
(1). Suppose $s=f(t)\in S$ is a weakly $S$-element of $N_{2}$. Let
$(r_{1},f(r_{1})+j)\in R_{1}\Join^{f}J$ and $(m_{1},\varphi(m_{1})+m_{2})\in
M_{1}\Join^{\varphi}JM_{2}$ such that
\[
(0,0)\neq(r_{1},f(r_{1})+j)(m_{1},\varphi(m_{1})+m_{2})\in\overline{N_{2}%
}^{\varphi}%
\]

Then $(f(r_{1})+j)(\varphi(m_{1})+m_{2})\in N_{2}$. If $(f(r_{1}%
)+j)(\varphi(m_{1})+m_{2})\neq0$, then the result follows as in the proof of
(1) in Theorem \ref{Amalg2}. Suppose $(f(r_{1})+j)(\varphi(m_{1})+m_{2})=0$ so
that $r_{1}m_{1}\neq0$. Then by assumption, there exists $s^{\prime
}=f(t^{\prime})\in S$ such that $s^{\prime}(f(r_{1})+j)\in(N_{2}:_{R_{2}}%
M_{2})$ or $s^{\prime}(\varphi(m_{1})+m_{2})\in N_{2}$. It follows clearly
that $(t^{\prime},s^{\prime})(r_{1},f(r_{1})+j)\in(\overline{N_{2}}^{\varphi
}:_{R_{1}\Join^{f}J}M_{1}\Join^{\varphi}JM_{2})$ or $(t^{\prime},s^{\prime
})(m_{1},\varphi(m_{1})+m_{2})\in\overline{N_{2}}^{\varphi}$. Hence,
$(tt^{\prime},ss^{\prime})$ is a weakly $\overline{S}^{\varphi}$-element of
$\overline{N_{2}}^{\varphi}$. Conversely, let $(t,f(t)+j)=(t,s)$ be a weakly
$\overline{S}^{\varphi}$-element of $\overline{N_{2}}^{\varphi}$. Let
$r_{2}=f(r_{1})\in R_{2}$ and $m_{2}=f(m_{1})\in M_{2}$ such that $0\neq
r_{2}m_{2}\in N_{2}$. Then $(r_{1},r_{2})\in R_{1}\Join^{f}J$ and
$(m_{1},m_{2})\in M_{1}\Join^{\varphi}JM_{2}$ with $(0.0)\neq(r_{1}%
,r_{2})(m_{1},m_{2})\in\overline{N_{2}}^{\varphi}$. Hence, either
$(t,s)(r_{1},r_{2})\in(\overline{N_{2}}^{\varphi}:_{R_{1}\Join^{f}J}M_{1}%
\Join^{\varphi}JM_{2})$ or $(t,s)(m_{1},m_{2})\in\overline{N_{2}}^{\varphi}$.
In the first case, for all $m=\varphi(m^{\prime})\in M_{2}$, $(tr_{1}%
,sr_{2})(m^{\prime},m)\in\overline{N_{2}}^{\varphi}$. Hence, $sr_{2}m\in
N_{2}$ and then $sr_{2}\in(N_{2}:_{R_{2}}M_{2})$. In the second case, we have
$sm_{2}\in N_{2}$ and so $s$ is a weakly $S$-element of $N_{2}$. Now, let
$r_{1}\in R_{1}$, $m_{1}\in M_{1}$, $m_{2}\in JM_{2}$, $j\in J$ with
$(f(r_{1})+j)(\varphi(m_{1})+m_{2})=0$ and suppose $r_{1}m_{1}\neq0$. Then
$(0,0)\neq(r_{1},f(r_{1})+j)(m_{1},\varphi(m_{1})+m_{2})\in\overline{N_{2}%
}^{\varphi}$ and so $(t,s)(r_{1},f(r_{1})+j)\in(\overline{N_{2}}^{\varphi
}:_{R_{1}\Join^{f}J}M_{1}\Join JM_{2})$ or $(t,s)(m_{1},\varphi(m_{1}%
)+m_{2})\in\overline{N_{2}}^{\varphi}$. Hence, clearly, either $s(f(r_{1}%
)+j)\in(N_{2}:_{R_{2}}M_{2})$ or $s(\varphi(m_{1})+m_{2})\in N_{1}$ and the
result follows by contrapositive.

(2) Suppose $(t,f(t)+j)=(t,s)$ is a weakly $\overline{S}^{\varphi}$-element of
$\overline{N_{2}}^{\varphi}$. Let $r_{2}=f(r_{1})\in R_{2}$, $m_{2}\in M_{2}$
such that $0\neq r_{2}m_{2}\in(N_{2}:_{M_{2}}J)$. Then $r_{2}Jm_{2}\subseteq
N_{2}$ and so for all $j\in J$, we have $(r_{1},r_{2})(0,jm_{2})\in
\overline{N_{2}}^{\varphi}$. If $j\neq0$ and $(r_{1},r_{2})(0,jm_{2})=(0,0)$,
then $r_{2}jm_{2}=0$ and so $r_{2}m_{2}=0$ as $Z_{R_{2}}(N_{2})\cap J=\left\{
0\right\}  $, a contradiction. Thus, for all $j\neq0$, $(r_{1},r_{2}%
)(0,jm_{2})\neq(0,0)$. By assumption and similar to the proof of (2) of
Theorem \ref{Amalg2}, we have for all $j\neq0$, either $sr_{2}jm_{2}\in N_{2}$
or $(t,s)(0,jm_{2})\in\overline{N_{2}}^{\varphi}$ for all $m_{2}\in M_{2}$.
Thus, $sr_{2}\in(N_{2}:_{R_{2}}JM_{2})=((N_{2}:_{M_{2}}J):_{R_{2}}M_{2})$ or
$sm_{2}\in(N_{2}:_{M_{2}}J)$ and we are done.
\end{proof}

\begin{corollary}
\label{ca3}Consider the $(R_{1}\Join^{f}J)$-module $M_{1}\Join^{\varphi}%
JM_{2}$ defined as in Theorem \ref{Amalg} where $f$ and $\varphi$ are
epimorphisms. If $N_{2}$ is a submodule of $M_{2}$, then
\end{corollary}

\begin{enumerate}
\item $\overline{N_{2}}^{\varphi}$ is a weakly prime submodule of $M_{1}%
\Join^{\varphi}JM_{2}$ if and only if $N_{2}$ is a weakly prime submodule of
$M_{2}$ and for $r_{1}\in R_{1}$, $m_{1}\in M_{1}$, $m_{2}\in JM_{2}$, $j\in
J$ with $(f(r_{1})+j)(\varphi(m_{1})+m_{2})=0$ but $(f(r_{1})+j)\notin
(N_{2}:_{R_{2}}M_{2})$ and $(\varphi(m_{1})+m_{2})\notin N_{2}$, then
$r_{1}m_{1}=0$.

\item If $\overline{N_{2}}^{\varphi}$ is a weakly prime submodule of
$M_{1}\Join^{\varphi}JM_{2}$, $J\nsubseteq(N_{2}:_{R_{2}}M_{2})$ and
$Z_{R_{2}}(N_{2})\cap J=\left\{  0\right\}  $, then $(N_{2}:_{M_{2}}J)$ is a
weakly prime submodule of $M_{2}$.
\end{enumerate}

\begin{corollary}
\label{Dup}Let $N$ be a submodule of an $R$-module $M$, $J$ an ideal of $R$
and $S$ a multiplicatively closed subset of $R$.
\end{corollary}

\begin{enumerate}
\item $N\Join J$ is an $(S\Join J)$-prime submodule of $M\Join J$ if and only
if $N$ is an $S$-prime submodule of $M$.

\item $N\Join J$ is a weakly $(S\Join J)$-prime submodule of $M\Join J$ if and
only if $N$ is a weakly $S$-prime submodule of $M$ and for $r\in R$, $m\in M$
with $rm=0$ but $sr\notin(N:_{R_{1}}M)$ and $sm\notin N$ for all $s\in S$,
then $(r+j)m^{\prime}=0$ for every $j\in J$ and $m^{\prime}\in M$ where
$(m,m^{\prime})\in M\Join J.$
\end{enumerate}

\begin{corollary}
Let $N$ be a submodule of an $R$-module $M$, $J$ an ideal of $R$ and $S$ a
multiplicatively closed subset of $R$.
\end{corollary}

\begin{enumerate}
\item $N$ is an $S$-prime submodule of $M$ if and only if $\overline{N}$ is an
$\overline{S}$-prime submodule of $M\Join J$.

\item If $\overline{N}$ is an $\overline{S}$-prime submodule of $M\Join J$ and
$(N:_{R}JM)\cap S=\phi$, then $(N:_{M}J)$ is an $S$-prime submodule of $M$.
\end{enumerate}

\begin{corollary}
\label{Dup2}Let $N$ be a submodule of an $R$-module $M$, $J$ an ideal of $R$
and $S$ a multiplicatively closed subset of $R$.
\end{corollary}

\begin{enumerate}
\item $\overline{N}$ is a weakly $\overline{S}$-prime submodule of $M\Join J$
if and only if $N$ is a weakly $S$-prime submodule of $M$ and for $r\in R$,
$m\in M$, $m^{\prime}\in JM$, $j\in J$ with $(r+j)(m+m^{\prime})=0$ but
$s(r+j)\notin(N:_{R}M)$ and $s(m+m^{\prime})\notin N$ for all $s\in S$, then
$rm=0$.

\item If $\overline{N}$ is a weakly $\overline{S}$-prime submodule of $M\Join
J$, $(N:_{M}J)\cap S=\phi$ and $Z_{R}(N)\cap J=\left\{  0\right\}  $, then
$(N:_{M}J)$ is a weakly $S$-prime submodule of $M$.
\end{enumerate}

In the following example, we show that in general $N$ is a weakly $S$-prime
submodule of $M$ does not imply $N\Join J$ is a weakly $(S\Join J)$-prime
submodule of $M\Join J$.

\begin{example}
\label{ex2}Consider the $%
\mathbb{Z}
$-submodule $N=0\times\left\langle \bar{0}\right\rangle $ of $M=%
\mathbb{Z}
\times%
\mathbb{Z}
_{6}$ and let $J=2%
\mathbb{Z}
$. Then $N$ is a weakly prime submodule of $M$. Now
\[
M\Join J=\left\{  (m,m^{\prime})\in M\times M:m-m^{\prime}\in JM=2%
\mathbb{Z}
\times\left\langle \bar{2}\right\rangle \right\}
\]
and
\[
N\Join J=\left\{  (n,m)\in N\times M:n-m\in2%
\mathbb{Z}
\times\left\langle \bar{2}\right\rangle \right\}
\]
If we consider $(2,4)\in%
\mathbb{Z}
\Join J$ and $((0,\bar{3}),(0,\bar{1}))\in M\Join J$, then we have
$(2,4).((0,\bar{3}),(0,\bar{1}))=((0,\bar{0}),(0,\bar{4}))\in N\Join J$. But
we have $(2,4)\notin((N\Join J):_{%
\mathbb{Z}
\Join I}(M\Join J))$ as for example $(2,4)((2,\bar{2}),(0,\bar{0}))\notin
N\Join J$ and $((0,\bar{3}),(0,\bar{1}))\notin N\Join J$. Thus, $N\Join J$ is
not a weakly prime submodule of $M\Join J$.
\end{example}

We note that the condition in the reverse implication of Corollary \ref{Dup}
(2) does not hold in example \ref{ex2}. For example, if we take $r=2$ and
$m=(0,\bar{3})\in M$, then clearly, $rm=0$, $r\notin(N:_{R}M)=0$ and $m\notin
N$ but for $m^{\prime}=(0,\bar{2})\in JM=2%
\mathbb{Z}
\times\left\langle \bar{2}\right\rangle $, we have $(r+0)m^{\prime}\neq0$.

Also, if the condition in the reverse implication of Corollary \ref{Dup2} (1)
does not hold, then we may find a weakly $S$-prime submodule $N$ of $M$ such
that $\overline{N}$ is not a weakly $\overline{S}$-prime submodule of $M\Join
J$.

\begin{example}
Consider $N$, $M$ and $J$ as in Example \ref{ex2}. If we consider $(2,4)\in%
\mathbb{Z}
\Join J$ and $((0,\bar{1}),(0,\bar{3}))\in M\Join J$, then we have
$(2,4).((0,\bar{1}),(0,\bar{3}))=\overline{N}$. But $(2,4)\notin(\bar{N}:_{%
\mathbb{Z}
\Join I}(M\Join J))$ and $((0,\bar{1}),(0,\bar{3}))\notin\overline{N}$. Thus,
$\bar{N}$ is not a weakly prime submodule of $M\Join J$.
\end{example}

\end{document}